\documentclass[a4paper,11pt]{amsart}

\usepackage{amsfonts}
\usepackage{amssymb}
\usepackage{amsmath}

\usepackage{enumerate}

\usepackage[top=4cm,bottom=4cm,right=4cm,left=2cm]{geometry}

\usepackage[abbrev,msc-links]{amsrefs}

\usepackage[utf8]{inputenc}
\usepackage[T1]{fontenc}

\usepackage{lmodern}

\usepackage{microtype}

\usepackage{backref}

\newtheorem{theorem}{Theorem}
\newtheorem{proposition}{Proposition}[section]
\newtheorem{lemma}[proposition]{Lemma}

\newcommand{\abs}[1]{\lvert #1 \rvert }
\newcommand{\bigabs}[1]{\bigl\lvert #1 \bigr\rvert }
\newcommand{\Bigabs}[1]{\Bigl\lvert #1 \Bigr\rvert }
\newcommand{\biggabs}[1]{\biggl\lvert #1 \biggr\rvert }
\newcommand{\norm}[1]{\lVert#1\rVert}
\newcommand{\dualprod}[2]{\langle #1, #2 \rangle}
\newcommand{\lt}{\leqslant}
\newcommand{\gt}{\geqslant}
\newcommand{\du}{\,\mathrm{d}}
\newcommand{\R}{\mathbb{R}}
\newcommand{\N}{\mathbb{N}}
\newcommand{\charfun}[1]{\chi_{#1}}
\DeclareMathOperator{\dive}{div}
\DeclareMathOperator{\curl}{curl}
\DeclareMathOperator{\capa}{cap}
\DeclareMathOperator{\dist}{dist}
\DeclareMathOperator{\diam}{diam}
\DeclareMathOperator{\inter}{int}
\DeclareMathOperator{\supp}{supp}
\newcommand{\st}{\,:\,}

\newcommand{\weakto}{\rightharpoonup}
\newcommand{\goesto}{\rightarrow}
\newcommand{\nehari}{\mathcal{N}_{\varepsilon}}

\allowdisplaybreaks 

\begin{document}

\title[Desingularization of vortices]{Desingularization of vortex rings and shallow water vortices by semilinear elliptic problem}
\author{S\'ebastien de Valeriola}
\address{Universit\'e catholique de Louvain\\
Institut de Recherche en Math\'ematique et Physique (IRMP)\\
Chemin du Cyclotron 2 bte L7.01.01\\
1348 Louvain-la-Neuve\\
Belgium}
\email{Sebastien.deValeriola@uclouvain.be}
\author{Jean Van Schaftingen}
\address{Universit\'e catholique de Louvain\\
Institut de Recherche en Math\'ematique et Physique (IRMP)\\
Chemin du Cyclotron 2 bte L7.01.01\\
1348 Louvain-la-Neuve\\
Belgium}
\email{Jean.VanSchaftingen@uclouvain.be}
\date{\today }

\begin{abstract}
Steady vortices for the three-dimensional Euler equation for inviscid incompressible flows and for the shallow water equation are constructed and showed to tend asymptotically to singular vortex filaments. 
The construction is based on a study of solutions to the semilinear elliptic problem
\[
\left\{
\begin{aligned}
  - \dive \Bigl(\frac{\nabla u_{\varepsilon}}{b}\Bigr) & = \frac{1}{\varepsilon^2} b f(u_{\varepsilon} - \log \tfrac{1}{\varepsilon} q) & & \text{in \(\Omega\)},\\
  u_\varepsilon & = 0 & & \text{on \(\partial \Omega\)},
\end{aligned}
\right.
\]
for small values of \(\varepsilon > 0\).
\end{abstract}



\maketitle

\tableofcontents

\section{Introduction and main results}

\subsection{Statement of the problem}
In an inviscid incompressible flow, the velocity field \(\mathbf{v}\) and static pressure field \(p\) are governed by the Euler equations
\begin{equation*}
\left\{
\begin{aligned}
\dive \mathbf{v} &= 0, \\
\partial_t \mathbf{v} + (\mathbf{v} \cdot \nabla) \mathbf{v} & = -\nabla p.
\end{aligned}
\right.
\end{equation*}
The conservation of momentum equation can be rewritten in terms of the vorticity \(\boldsymbol{\omega} = \curl \mathbf{v}\) as
\[
 \partial_t \mathbf{v} + \boldsymbol{\omega} \times \mathbf{v} = -\nabla \Bigl(p + \frac{\abs{\mathbf{v}}^2}{2} \Bigr).
\]
The quantities \(\frac{\abs{\mathbf{v}}^2}{2}\) and \(p + \frac{\abs{\mathbf{v}}^2}{2}\) are called \emph{dynamic pressure} and \emph{total pressure}.
In regions where the vorticity vanishes \(\boldsymbol{\omega} = 0\), the flow is called irrotational and the  equations reduce to the Bernoulli equation.
In other cases, one can study flows which are irrotational outside of a vortex core.

In 1858, Helmoltz has studied the motion of vortex rings, which are toroidal regions in which the vorticity is concentrated \cite{He1858}. 
The circulation \(\kappa\) of a vortex is the circulation integral \(\int_\Gamma \mathbf{v} \cdot \mathbf{t}\) for any oriented curve \(\Gamma\) with tangent vector field \(\mathbf{t}\) that encircles the vorticity region once.
Kelvin and Hick have showed that if the vortex ring has radius \(r_*\), if its cross-section \(\varepsilon\) is small and if its circulation is \(\kappa\), then the vortex ring moves at the velocity \citelist{\cite{La32}*{art. 163 (7), p. 241}\cite{Ke1910}*{67}}
\begin{equation}
\label{eqKelvinHick}
 \frac{\kappa}{4 \pi r_*} \Bigl(\log \frac{8 r_*}{\varepsilon} - \frac{1}{4}\Bigr).
\end{equation}
In this initial study of vortex motion, the flows were not steady flows; as the velocity is merely asymptotically constant in the vortex, one does not expect the vortex ring to preserve its shape.
After the works of Helmholtz, Kelvin \cite{Ke1910} interested himself in this problem and stated a variational principle for steady vortex flows.
In 1894, Hill has given an explicit translating flow of the Euler equation whose vorticity is concentrated \emph{inside a ball} \cite{Hi1894}. 

These works bring the question whether it is possible to construct flows whose vorticity is supported in an arbitrarily small toroidal region.
Fraenkel has given a first positive answer by constructing for small \(\varepsilon > 0\) a family of steady flows whose vortex cross section is of the order of \(\varepsilon\) and whose velocity satisfy asymptotically \eqref{eqKelvinHick} \citelist{\cite{Fr70}\cite{Fr73}}.
His approach consists in first noting that since the flow is incompressible in the whole space, it is possible to write \(\mathbf{v} = \curl \boldsymbol{\psi}\) where \(\boldsymbol{\psi}\) is a velocity vector potential. Moreover, since the flow should be axisymmetric, the vector potential \(\boldsymbol{\psi}\)  can be written in terms of the Stokes stream function \(\psi\) in cylindrical coordinates \((r, \theta, z)\),
\[
  \boldsymbol{\psi} (r, \theta, z) = \psi (r, z) \frac{\mathbf{e}_\theta}{r};
\]
the associated velocity field is
\[
 \mathbf{v} (r, \theta, z) = \frac{1}{r} \Bigl( - \frac{\partial \psi}{\partial z} \mathbf{e}_r + \frac{\partial \psi}{\partial r} \mathbf{e}_z \Bigr)
\]
and the associated vorticity is
\[
 \boldsymbol{\omega} (r, \theta, z) 
= - \Bigl(\frac{\partial}{\partial r} \Bigl( \frac{1}{r} \frac{\partial \psi}{\partial r}\Bigr) + \frac{\partial}{\partial z} \Bigl( \frac{1}{r} \frac{\partial \psi}{\partial z}\Bigr) \Bigr) \mathbf{e}_\theta.
\]
The key point is to note that if \(\boldsymbol{\omega} = r f (\psi) \mathbf{e}_\theta\) for some function \(f : \R \to \R\) and \(F' = f\), then 
\[
 \boldsymbol{\omega} \times \mathbf{v} = - \nabla \bigl(F (\psi) \bigr),
\]
that is, \(\mathbf{v}\) is a stationary solution of the incompressible Euler equation with \(p = F (\psi) - \frac{\abs{\mathbf{v}}^2}{2}\).
The problem is thus reduced to a study of the \emph{semilinear elliptic problem}
\begin{equation}
\label{eqSemiLinear}
   - \Bigl(\frac{\partial}{\partial r} \Bigl( \frac{1}{r} \frac{\partial \psi}{\partial r}\Bigr) + \frac{\partial}{\partial z} \Bigl( \frac{1}{r} \frac{\partial \psi}{\partial z}\Bigr) \Bigr) = r f (\psi)
\end{equation}
Fraenkel constructed solutions to this problem by a variant of the implicit function theorem.
We call this construction the \emph{stream-function method} in contrast with the \emph{vorticity method} developed by Friedman and Turkington in which the vorticity \(\boldsymbol{\omega}\) instead of the stream function is a solution of a variational problem \cite{FrTu81} (see also \citelist{ \cite{BeBr80}\cite{Fr82}\cite{Bu87a}\cite{Bu87b}\cite{Bu03}\cite{BaBu01}\cite{BuPr04}}).
The stream function method together with an implicit function argument was used to construct vortex rings close to Hill's spherical vortex \citelist{\cite{No72}\cite{No73}\cite{Bu97}}.

Afterwards, vortex rings were constructed with the stream function method by constructing solutions to \eqref{eqSemiLinear} by minimization under constraint; their asymptotics could not be studied precisely because of the presence of a Lagrange multiplier in the nonlinearity \(f\) \citelist{\cite{BeFr74}\cite{BeFr80}}.
The asymptotics could be studied precisely by letting the flux diverge \cite{Ta94}.
By using the mountain pass theorem of Ambrosetti and Rabinowitz \cite{AmRa73}, Ambrosetti and Mancini, Ni, and Ambrosetti and Struwe have constructed solutions for a given \(f\) \citelist{\cite{Ni80}\cite{AmMa81}\cite{AmSt89}}.
The asymptotics of a family \((\psi_\varepsilon)\) of these solutions have been studied by Ambrosetti and Yang for a family \(f_\varepsilon (s) = \frac{1}{\varepsilon^2} (s)_+^p\)  \cite{Ya95}.
However, their result did not prevent the circulation of the vortex to go to \(0\) and, according to our present work, it does go to \(0\) so that the limiting object are degenerate vortex rings with \emph{vanishing radius} and \emph{vanishing circulation}. 

Finally, we would like to mention that it is possible to study the asymptotics of the motion of vortices in the nonsteady case \cite{BeCaMa00}.

All the results that we have mentioned above have counterparts in the study of \emph{vortex pairs} for the two-dimensional Euler equation \citelist{\cite{No75}\cite{BeFr80}\cite{AmYa90}\cite{Ya91}\cite{LiYaYa05}}. 
In particular, Smets and Van Schaftingen have showed that in order to obtain nonvanishing asymptotic circulation one could, instead of imposing fixed boundary conditions \(\psi_\varepsilon = \psi_0 + o (1)\) at infinity, impose boundary conditions depending on \(\varepsilon\):  \(\psi_\varepsilon = \psi_0 - \frac{\kappa}{2\pi} \log \tfrac{1}{\varepsilon} + o (1)\) at infinity \cite{SmVS10}. Physically, this takes into account that the total flow between the two vortices should blow up as the logarithm of the diameter of the vortex core. They have obtained a desingularization result for solutions constructed by variational methods; solutions to the same problem where also obtained by Lyapunov--Schmidt reduction argument \citelist{\cite{CaLiWe2012a}\cite{CaLiWe2012b}}.

\subsection{Vortex rings for the Euler equation}
In the present work, following the idea of Smets and Van Schaftingen, we consider the semilinear elliptic problem
\begin{equation}
\label{problemSemilinearRing}
\left\{
\begin{aligned}
  - \dive \frac{1}{r} \nabla \psi_\varepsilon &= \frac{r}{\varepsilon^2} (\psi_\varepsilon)_+^p & &\text{in \(\R^2_+\),}\\
  \frac{\psi_\varepsilon}{\psi_0} & \to \log \tfrac{1}{\varepsilon} & & \text{at \(\infty\)}.
\end{aligned}
\right.
\end{equation}
We study the asymptotic behaviour of its solutions. 
Even if the semilinear elliptic problem is similar to the corresponding problem for the two-dimensional, the asymptotics of the solutions are quite different. For instance, whereas in \cite{SmVS10} the localization of concentration points is governed by a renormalized enery which appears as a second term in the asymptotics, in the present work the solution concentrates at minimizers of the leading term. 

As a consequence of these asymptotics, we obtain first a desingularization result in the whole space.

\begin{theorem}
\label{theoremVortexRingSpace}
For every \(W > 0\) and \(\kappa > 0\), there exists a family of steady flows \((\mathbf{v}_{\varepsilon}, p_\varepsilon) \in C^1 (\R^3)\) for the Euler equations in \(\R^3\) that are axisymmetric around \(\mathbf{e}_3\) and such that 
the vortex core \(\supp \curl \mathbf{v}_{\varepsilon}\) is a topological torus, the circulation of the vortex ring is \(\kappa_\varepsilon\)
and for every \(\varepsilon \in (0, 1)\),
\begin{align*}
 \mathbf{v}_\varepsilon & \to  - W \log \tfrac{1}{\varepsilon} \mathbf{e}_z & & \text{at \(\infty\)}.
\end{align*}
Moreover, one has 
\begin{gather*}
 \lim_{\varepsilon \to 0} \kappa_\varepsilon = \kappa,\\
 \lim_{\varepsilon \to 0} \dist_{C_{r_*}} (\supp \curl \mathbf{v}_\varepsilon) = 0,\\
 c \varepsilon \lt \sigma (\supp \curl \mathbf{v}_\varepsilon) \lt C \varepsilon,
\end{gather*}
for some constants \(0 < c < C\) and 
\[
  r_* = \frac{\kappa}{4 \pi W}.
\]
\end{theorem}

Here, the cross-section of a set \(A \subset \R^3\) axisymmetric around \(\mathbf{e}_3\) is 
\[
 \sigma (A) = \sup \bigl\{\delta_3 (x, y) \st x, y \in A\bigr\},
\]
where the axisymmetric distance is defined by
\[
 \delta_3 (x, y) = \inf \bigl\{ \abs{x - R (y)} \st R \text{ is a rotation around \(\mathbf{e}_3\)}\bigr\},
\]
\(C_{r}\) is a circle of radius \(r\) in a plane perpendicular to \(\mathbf{e}_3\)
and the asymmetric distance is 
\[
 \dist_{C_r} (A) =  \sup_{x \in A} \inf_{y \in C_r} \abs{x - y}.
\]

Our construction and our study of asymptotics are quite flexible. For example, we can study vortex rings in a cylinder.

\begin{theorem}
\label{theoremVortexRingCylinder}
For every \(W > 0\) and \(\kappa > 0\), there exists a family of steady flows \((\mathbf{v}_{\varepsilon}, p_\varepsilon) \in C^1 (B_1 \times \R)\) for the Euler equations in \(B_1 \times \R\) that are axisymmetric around \(\mathbf{e}_3\) and such that 
\begin{align*}
\mathbf{v}_\varepsilon \cdot \mathbf{n} &= 0, & & \text{on \(\partial B_1 \times \R^2\)},\\
 \mathbf{v}_\varepsilon & \to  - W \log \tfrac{1}{\varepsilon} \mathbf{e}_z & & \text{at \(\infty\)},
\end{align*}
the vortex core \(\supp \curl \mathbf{v}_{\varepsilon}\) is a topological torus, the circulation of the vortex is \(\kappa_\varepsilon\).
Moreover, one has 
\begin{gather*}
 \lim_{\varepsilon \to 0} \kappa_\varepsilon = \kappa,\\
 \lim_{\varepsilon \to 0} \dist_{C_r} (\supp \curl \mathbf{v}_\varepsilon) = 0,\\
 \lim_{\varepsilon \to 0} \frac{\log \sigma (\supp \curl \mathbf{v}_\varepsilon)}{\log \varepsilon} = 1,
\end{gather*}
and 
\[
 r_* = 
\left\{
\begin{aligned}
     &\frac{\kappa}{4 \pi W} & &\text{if \(\kappa < 4 \pi W\)},\\
     &1 & &\text{if \(\kappa \gt 4 \pi W\)}.
\end{aligned}
\right.
\]
\end{theorem}

Burton has constructed similar vortex rings in a cylinder, but he did not study their asymptotics \cite{Bu87a}.

If \(\kappa > 4 \pi W\), the velocity \(W \log \tfrac{1}{\varepsilon} \) of the vortex ring is less than  predicted by the Kelvin--Hick formula \eqref{eqKelvinHick}.
We do not study in detail this phenomenon in the present work, but we think that it might be explained by an interaction with the boundary that reduces the velocity by
\[
    \frac{\kappa}{4 \pi \dist (\supp \curl \mathbf{v}_\varepsilon, \partial B(0, 1) \times \R)},
\]
similar to the contribution of the boundary for the two-dimensional Euler equation \cite{SmVS10}. 
This could also explain why the asymptotics of \(\sigma (\supp \curl \mathbf{v}_\varepsilon)\) are less sharp than those of theorem~\ref{theoremVortexRingSpace}.

\bigskip

Similarly we can study vortex rings outside a ball.

\begin{theorem}
\label{theoremVortexRingOutsideBall} 
For every \(W > 0\) and \(\kappa > 0\), there exists a family of steady flows \((\mathbf{v}_{\varepsilon}, p_\varepsilon) \in C^1 (\R^3 \setminus B_1)\) for the Euler equations in \(\R^3\) that are axisymmetric around \(\mathbf{e}_3\) and such that 
the vortex core \(\supp \curl \mathbf{v}_{\varepsilon}\) is a topological torus, the circulation of the vortex ring is \(\kappa_\varepsilon\)
and 
\begin{align*}
 \mathbf{v}_\varepsilon \cdot \mathbf{n} & = 0 & & \text{on \(\partial B_1\)},\\
 \mathbf{v}_\varepsilon & \to  - W \log \tfrac{1}{\varepsilon} \mathbf{e}_z & & \text{at \(\infty\)}.
\end{align*}
Moreover, one has 
\begin{gather*}
 \lim_{\varepsilon \to 0} \kappa_\varepsilon = \kappa,\\
 \lim_{\varepsilon \to 0} \dist_{C_{r_*}} (\supp \curl \mathbf{v}_\varepsilon) = 0,\\
\frac{\log \sigma (\supp \curl \mathbf{v}_\varepsilon)}{\log \varepsilon} = 1,
\end{gather*}
for \(r_*\) such that 
\[
  \mathbf{v} (r_*, 0) = - \frac{\kappa}{4 \pi W} \mathbf{e}_z,
\]
where \(\mathbf{v}_0 : \R^3 \setminus B_1\) is the irrotational flow outside \(B_1\) with velocity \(W\) at infinity:
\[
\left\{
\begin{aligned}
  \dive \mathbf{v}_0 &= 0 & & \text{in \(\R^3 \setminus B_1\)},\\
  \curl \mathbf{v}_0 &= 0 & & \text{in \(\R^3 \setminus B_1\)},\\
  \mathbf{v}_0 \cdot \mathbf{n} & = 0 & & \text{on \(\partial B_1\)},\\
  \mathbf{v}_0 & \to - W \mathbf{e}_z & & \text{at \(\infty\)}.
\end{aligned}
\right.
\]
\end{theorem}

The main difference in the proof of theorem~\ref{theoremVortexRingOutsideBall} is that the existence relies on a concentration-compactness argument \citelist{\cite{Li84}\cite{Ra92}}.

It is moreover possible to extend these results in some sense to a general outside domain.
\begin{theorem}
\label{theoremVortexRingOutsideCompact} 
Let \(K \subset \R^3\) be compact, connected and symmetric under rotations around \(\mathbf{e}_3\).
For every \(W > 0\) and for every \(\psi : \R^2_+ \to (-\infty, 0)\) such that \(\mathbf{v}_0 = \curl (\psi \mathbf{e}_\theta/r)\) solves
\[
\left\{
\begin{aligned}
  \dive \mathbf{v}_0 &= 0 & & \text{in \(\R^3 \setminus K\)},\\
  \curl \mathbf{v}_0 &= 0 & & \text{in \(\R^3 \setminus K\)},\\
  \mathbf{v}_0 \cdot \mathbf{n} & = 0 & & \text{on \(\partial K\)},\\
  \mathbf{v}_0 & \to - W \mathbf{e}_3 & & \text{at \(\infty\)},
\end{aligned}
\right.
\]
there exists a family of steady flows \((\mathbf{v}_{\varepsilon}, p_\varepsilon) \in C^1 (\R^3 \setminus B_1)\) for the Euler equations in \(\R^3\) that are axisymmetric around \(\mathbf{e}_3\) and such that 
the vortex core \(\supp \curl \mathbf{v}_{\varepsilon}\) is a topological torus, the circulation of the vortex ring is \(\kappa_\varepsilon\)
and
\begin{align*}
 \mathbf{v}_\varepsilon &\to - W \mathbf{e}_3 & & \text{at \(\infty\)},\\
 \mathbf{v}_\varepsilon \cdot \mathbf{n} & = 0 & & \text{on \(\partial B_1\)},
\end{align*}
Moreover, if \((a_{\varepsilon})_{\varepsilon > 0} = ( (r_\varepsilon, z_\varepsilon))_{\varepsilon > 0}\) is a family such that \(\curl \mathbf{v}_\varepsilon (a_{\varepsilon}) \ne 0\), 
\begin{gather*}
 \lim_{\varepsilon \to 0} \frac{r_\varepsilon}{\psi (r_\varepsilon, z_\varepsilon)} \kappa_\varepsilon = - 2 \pi,\\
 \lim_{\varepsilon \to 0} \frac{\psi (r_\varepsilon, z_\varepsilon)^2}
{r_\varepsilon} = \inf_{(r, \theta, z) \R^3 \setminus K} \frac{\psi (r, z)^2}{r},\\
 \lim_{\varepsilon \to 0} \frac{\log \sigma (\supp \curl \mathbf{v}_\varepsilon)}{\log \varepsilon} = 1.
\end{gather*}
\end{theorem}

Note that given \(W > 0\), there are infinitely many \(\psi\) that satisfy the equation and the sign assumption (see lemma~\ref{lemmaLinear}), so that there are several families concentrating at different points with different asymptotic circulations.

In the case where \((r, z) \mapsto \frac{\psi (r, z)^2}{r}\) achieves its maximum at a unique interior point \((r_*, z_*)\), one has \((r_\varepsilon, z_\varepsilon) \to (r_*, z_*)\), and 
\begin{equation}
\label{eqVelocityRingCompact}
   \log \tfrac{1}{\varepsilon} \mathbf{v}_0 (r_*, z_*) = \frac{1}{r_*} \nabla \psi (r_*, z_*) \times \mathbf{e}_z = \frac{1}{2}\frac{\psi (r_*, z_*)}{r_*^2} \mathbf{e}_r \times \mathbf{e}_\theta = - \log \tfrac{1}{\varepsilon} \frac{1}{4 \pi r_*}
\lim_{\varepsilon \to 0} \kappa_\varepsilon \; \mathbf{e}_z,
\end{equation}
in accordance with \eqref{eqKelvinHick}.

\subsection{Vortices for the shallow water equation}
The same technique allows us to desingularize vortices for the shallow water equation with vanishing Froude number \(\mathrm{Fr}\) in the so-called lake model. 
The horizontal velocity \(\mathbf{v}\), the height \(h\) and the depth \(b\) satisfy the system  \citelist{\cite{CaHoLe96}\cite{CaHoLe97}}:
\begin{equation}
\label{eqLake}
\left\{
\begin{aligned}
  \dive (b \mathbf{v}) & = 0\\
  \partial_t \mathbf{v} + \mathbf{v} \cdot \nabla \mathbf{v} &= - \nabla h.
\end{aligned} 
\right.
\end{equation}
Richardson has computed by the method of matched asymptotics the velocity  of a vortex of circulation \(\kappa\) at \(x_*\) to be formally \cite{Ri00}*{(5.1)}
\footnote{Richardson writes the asymptotics in terms of \(\Gamma = \frac{\kappa}{2 \pi}\) \cite{Ri00}*{(2.19)}}
\begin{equation}
\label{eqVelocityRichardson}
   (\nabla \log b (x_*)) \times \frac{\kappa \mathbf{e}_3}{4 \pi}\log \tfrac{1}{\varepsilon}  + O (1);
\end{equation}
in particular, a vortex follows an isobath (level set of the depth).

We want to exhibit this in the asymptotics of families of steady flows.
As previously, setting \(\omega = \curl \mathbf{v}\), the second equation becomes
\[
  \partial_t \mathbf{v} + \omega \times \mathbf{v} = - \nabla \Bigl(\frac{\abs{\mathbf{v}}^2}{2} + h\Bigr). 
\]
Taking a stream function \(\psi\), one can write \(\mathbf{v} = (\curl \psi)/b\) and observe that 
if \(\omega = f (\psi)\), then \(\mathbf{v}\) is a stationary solution with \(h = F (\psi) - \frac{\abs{\mathbf{v}}^2}{2}\).
We are thus interested in studying the asymptotics of solutions of 
\begin{equation}
\label{problemSemilinearShallow}
\left\{
 \begin{aligned}
    - \dive \frac{1}{b} \nabla \psi_\varepsilon & = \frac{b}{\varepsilon^2} (\psi_\varepsilon)_+^p & & \text{in \(\Omega\)},\\
   \psi_\varepsilon & = \log \tfrac{1}{\varepsilon} \psi_0 & & \text{on \(\partial \Omega\)}.
 \end{aligned}
\right.
\end{equation}

\begin{theorem}
\label{theoremLakeMaximumDepth}
Let \(\Omega \subset \R^2\) be bounded and open and let \(b \in C (\Bar{\Omega}) \cap C^{1, \alpha} (\Omega)\) for some \(\alpha \in (0, 1)\).
If \(\inf_\Omega b > 0\), then there exists a family of solutions \(\mathbf{v}_{\varepsilon} \in C^1 (\Omega; \R^2)\) and \(h_\varepsilon \in C^1 (\Omega)\) of 
\[
\left\{
\begin{aligned}
  \dive (b \mathbf{v}_\varepsilon) & = 0 & & \text{in \(\Omega\)},\\
  \mathbf{v}_\varepsilon \cdot \nabla \mathbf{v}_\varepsilon &= - \nabla h_\varepsilon & & \text{in \(\Omega\)},\\
  \mathbf{v}_\varepsilon \cdot \mathbf{n} & = 0 & & \text{on \(\partial \Omega\)}.
\end{aligned} 
\right.
\]
Moreover if \(\kappa_\varepsilon = \int_{\Omega} \curl \mathbf{v}_\varepsilon\) and \(\curl \mathbf{v}_\varepsilon (x_\varepsilon) \ne 0\), then 
\begin{gather*}
 \lim_{\varepsilon \to 0} \kappa_\varepsilon = \kappa,\\
  \lim_{\varepsilon \to 0} b (x_\varepsilon) = \sup_{\Omega} b,\\
 \lim_{\varepsilon \to 0} \frac{\log \diam \supp \curl \mathbf{v}_\varepsilon}{\log \varepsilon} = 0.
\end{gather*}
\end{theorem}

In particular, if \(\lim_{n \to \infty} x_{\varepsilon_n} = x_* \in \Bar{\Omega}\) for some sequence \((\varepsilon_n)_{n \in \N}\), then \(x_*\) is a maximum point of \(b\) on \(\Bar{\Omega}\). 
If \(x_* \in \Omega\), then \(\nabla (\log b) (x_*) = 0\) and the velocity given by \eqref{eqVelocityRichardson} vanishes.
If \(x_* \in \partial \Omega\), then \(\nabla (\log b)\) is normal to the boundary so that the velocity given by \eqref{eqVelocityRichardson} is tangential to the boundary and would lead the vortex to circulate around \(\partial \Omega\) in the orientation opposite to the vortex's orientation; there should however be, as for the two-dimensional Euler equation \cite{SmVS10}, an interaction of the vortex with the boundary that should give a compensating term 
\[
    \frac{\kappa}{4 \pi} \log \frac{1}{\dist (\supp \curl \mathbf{v}_\varepsilon, \partial \Omega)}.
\]
If \(b\) is constant, theorem~\ref{theoremLakeMaximumDepth} does not locate the vortex; the refined asymptotics for the Euler equation locate them at maxima of the Robin function of \(\Omega\) \cite{SmVS10}.

theorem~\ref{theoremLakeMaximumDepth} constructs vortices at stationary points. 
We can also desingularize vortices at other points by prescribing the boundary condition.
First we note that if \(\psi_0\) satisfies 
\[
 - \dive \Bigl( \frac{\nabla \psi_0}{b} \Bigr) = 0,
\]
then \(\mathbf{v}_0 = \curl \psi_0\) is an irrotational stationary solution of \eqref{eqLake}.

\begin{theorem}
\label{theoremLakeStrongWind}
Let \(\Omega \subset \R^2\) be bounded and open, let \(b \in C (\Bar{\Omega}) \cap C^{1, \alpha} (\Omega)\) for some \(\alpha \in (0, 1)\), let \(\psi_0 \in C^2 (\Omega) \cap C^1 (\Bar{\Omega})\) be such that 
\[
 - \dive \Bigl( \frac{\nabla \psi_0}{b} \Bigr) = 0
\]
and let \(\mathbf{v}_0 = \curl \psi_0\).
If \(\sup \psi_0 < 0\) and \(\inf_\Omega b > 0\), then there exists a family of solutions \(\mathbf{v}_{\varepsilon} \in C^1 (\Omega; \R^2)\) and \(h_\varepsilon \in C^1 (\Omega)\) of 
\[
\left\{
\begin{aligned}
  \dive (b \mathbf{v}_\varepsilon) & = 0 & & \text{in \(\Omega\)},\\
  \mathbf{v}_\varepsilon \cdot \nabla \mathbf{v}_\varepsilon &= - \nabla h_\varepsilon & & \text{in \(\Omega\)},\\
  \mathbf{v}_\varepsilon \cdot \mathbf{n} & = \mathbf{v}_0 \cdot \mathbf{n} \log \tfrac{1}{\varepsilon} & & \text{on \(\partial \Omega\)},
\end{aligned} 
\right.
\]
such that if \(\kappa_\varepsilon = \int_{\Omega} \curl \mathbf{v}_\varepsilon\) and  \(\curl \mathbf{v} (x_\varepsilon ) \ne 0\),
\begin{gather*}
 \lim_{\varepsilon \to 0} \frac{b (x_\varepsilon )}{\psi_0 (x_\varepsilon )} \kappa_\varepsilon = - 2 \pi,\\ 
  \lim_{\varepsilon \to 0} \frac{b (x_\varepsilon )}{\psi_0 (x_\varepsilon )^2} = \sup_{\Omega} \frac{b}{\psi_0{}^2},\\
 \lim_{\varepsilon \to 0} \frac{\log \diam \supp \curl \mathbf{v}_\varepsilon}{\log \varepsilon} = 0,
\end{gather*}
\end{theorem}

In particular, if \(x_{\varepsilon_n} \to x_* \in \Omega\), then \(x_*\) is a maximum point of \(b/\psi_0^2\) on \(\Omega\) and 
\[
\frac{\nabla \psi_0 (x_*)}{b (x_*)} = \frac{1}{2} \frac{\nabla b (x_*)}{b (x_*)^2} \psi_0 (x_0)
\]
so that, similarly to \eqref{eqVelocityRingCompact},
\[
 \log \tfrac{1}{\varepsilon} \mathbf{v}_0 (x_*) = - \log \tfrac{1}{\varepsilon}\frac{\bigl(\nabla (\log b) (x_*) \bigr)}{4 \pi} \times \bigl(\lim_{\varepsilon \to 0} \kappa_\varepsilon \ \mathbf{e}_3\bigr),
\]
which is consistent with Richardson's formula \eqref{eqVelocityRichardson}.

The sequel of the paper is organized as follows. In section~\ref{sectionConstruction} we give sufficient conditions for the existence of solutions to \eqref{problemSemilinearShallow} that include \eqref{problemSemilinearRing} as particular cases. Next we study in section~\ref{sectionAsymptotics} the asymptotics of families of least energy solutions to those equations. Finally, we show in section~\ref{sectionConstruction} how the sufficient conditions for existence and the asymptotics can be combined to prove the theorems of the present section. 

\section{Construction of solutions}
\label{sectionConstruction}

\subsection{Preliminaries}

In order to have homogeneous boundary conditions, we rewrite problem~\eqref{problemSemilinearRing} and \eqref{problemSemilinearShallow} by defining \(q = - \psi_0\), \(q_\varepsilon = (\log \tfrac{1}{\varepsilon}) q\) 
and \(u_\varepsilon = \psi_\varepsilon + q_\varepsilon\). We are thus interested in solving
\begin{equation} 
\label{problemP}
\tag{$\mathcal{P}$}
\left\{
\begin{aligned}
-\dive \Bigl(\frac{\nabla u_{\varepsilon}}{b}\Bigr) &=  \frac{b}{\varepsilon^2} (u_{\varepsilon} - q_\varepsilon)^p_+  & &\text{ in \(\Omega\)}, \\
u_{\varepsilon} & = 0 & &\text{ on \(\partial \Omega\)},
\end{aligned}
\right.
\end{equation}
for \(\Omega \subset \R^2\) open, \(b : \Omega \to \R\) and \(q : \Omega \to \R\) measurable functions and for some fixed \(p > 1\).

Solutions to \eqref{problemP} are critical points of the  functional
\begin{equation*}
\mathcal{E}_{\varepsilon} (u) = \frac{1}{2} \int_{\Omega} \frac{1}{b} \abs{\nabla u}^2 - \frac{1}{(p + 1)\varepsilon^2} \int_{\Omega} b \ (u-q_\varepsilon)^{p + 1}_+,
\end{equation*}
defined for \(u \in C^\infty_c(\Omega) = \{u \in C^{\infty}(\Omega) : \supp u \text{ is compact in \(\Omega\)}\}\).
A natural space for this functional is the completion \(H^1_0 (\Omega, b)\) of \(C^\infty_c (\Omega)\) with respect to the norm defined for \(u \in C^\infty_c(\Omega)\) by 
\begin{equation*}
\norm{u}^2_{H^1_0 (\Omega, b)} = \int_{\Omega} \frac{\abs{\nabla u}^2}{b}.
\end{equation*}
In general \(H^1_0 (\Omega, b)\) needs not to be a space of distributions; but whenever the functional \(\mathcal{E}_{\varepsilon}\) has a well-defined extension to \(H^1_0 (\Omega, b)\), this space will be a well-defined space of locally integrable functions.

If \(\mathcal{E}_{\varepsilon}\) is continuously Fr\'echet--differentiable on \(H^1_0 (\Omega, b)\), we have the useful computation:

\begin{lemma}
\label{lemmeGradientEnergie}
Let \(\varepsilon \in (0, 1)\). If \(\mathcal{E}_{\varepsilon} \in C^1 (H^1_0 (\Omega, b); \R)\) and \(q \gt 0\), then for every \(u \in H^1_0 (\Omega, b)\),
\begin{equation*}
\Bigl( \frac{1}{2} - \frac{1}{p+1} \Bigr) \int_{\Omega} \frac{\abs{\nabla u}^2}{b}  \lt \mathcal{E}_{\varepsilon} (u) - \frac{1}{p + 1} \dualprod{\mathcal{E}_{\varepsilon}' (u)}{u}.
\end{equation*}
\end{lemma}

\begin{proof}
For \(u \in H^1_0 (\Omega, b)\), we compute
\[
\begin{split}
 \mathcal{E}_{\varepsilon} (u) - \frac{1}{p + 1} \dualprod{\mathcal{E}_{\varepsilon}' (u)}{u} = \ & \Bigl( \frac{1}{2} - \frac{1}{p+1} \Bigr) \int_{\Omega} \frac{\abs{\nabla u}^2}{b}\\
& +\frac{1}{(p + 1) \varepsilon^2} \int_\Omega b \bigl((u - q_\varepsilon)^{p + 1}_+- (u - q_\varepsilon)_+^p u \bigr).
\end{split}
\]
The bound follows as \(q_\varepsilon \gt 0\) and thus \((u - q_\varepsilon)_+ \lt u\).
\end{proof}

The Nehari manifold associated to the problem \eqref{problemP} is defined as
\begin{equation*}
\nehari = \bigl\{ u \in H^1_0 (\Omega, b) \setminus \{ 0\} :  \dualprod{ \mathcal{E}_{\varepsilon}'(u)}{u } = 0 \bigr\}
\end{equation*}
and the infimum of the energy on this manifold is
\begin{equation*}
c_{\varepsilon} = \inf_{u \in \nehari} \mathcal{E}_{\varepsilon}(u).
\end{equation*}
It can be characterized as follows:

\begin{lemma}
\label{lemmaCritical}
Let \(\varepsilon \in (0, 1)\). If \(\mathcal{E}_{\varepsilon} \in C^1 (H^1_0 (\Omega, b); \R)\), \(q \gt 0\) and 
\[
  \lim_{u \to 0} \frac{\displaystyle\int_{\Omega} (u - q)_+^{p + 1} }{\displaystyle\int_{\Omega} \frac{\abs{\nabla u}^2}{b}} = 0,
\]
then
\[
 c_{\varepsilon} = \inf_{u \in \nehari} \mathcal{E}_{\varepsilon}(u) 
= \inf_{u \in H^1_0 (\Omega, b) \setminus \{0\}} \sup_{t \gt 0} \mathcal{E}_{\varepsilon} (t u)
= \inf_{\gamma \in \Gamma_\varepsilon} \max_{t \in [0, 1]} \mathcal{E}_{\varepsilon} \bigl( \gamma (t)\bigr),
\]
where
\[
   \Gamma_\varepsilon = \Bigl\{ \gamma \in C\bigl([0, 1]; H^1_0 (\Omega, b)\bigr) : \gamma (0) = 0 \text{ and } \mathcal{E}_{\varepsilon} \bigl(\gamma(1)\bigr) < 0 \Bigr\}.
\]
Moreover there exists a sequence \((u^n)_{n \in \N}\) such that \(\mathcal{E}_{\varepsilon}(u^n) \goesto c_\varepsilon\) and \(\mathcal{E}_{\varepsilon}'(u^n) \goesto 0\) in \((H^1_0 (\Omega, b))'\) as \(n \goesto \infty\). 
\end{lemma}

A sequence \((u^n)_{n \in \N}\) such that sequence \(\mathcal{E}_{\varepsilon}(u^n) \goesto c_\varepsilon\) and \(
\mathcal{E}_{\varepsilon}'(u^n) \goesto 0\) in \((H^1_0 (\Omega, b))'\) as \(n \goesto \infty\) is called a Palais-Smale sequence at the level \(c_{\varepsilon}\).

The equivalence between the different critical levels goes back to  Rabinowitz \citelist{\cite{Ra92}*{proposition~3.11}\cite{Wi96}*{theorem~4.2}}. 
The assumptions of lemma~\ref{lemmaCritical} do not fit into the existing results, but existing arguments still work.

\begin{proof}[Proof of lemma~\ref{lemmaCritical}]
For \(u \in \nehari\), and \(t \in [0, \infty)\), observe that 
\[
\begin{split}
 \mathcal{E}_{\varepsilon} (u)
& = \mathcal{E}_{\varepsilon} (t u) + \frac{1 - t^2}{2} \int_{\Omega} \frac{\abs{\nabla u}^2}{b} + \frac{1}{(p + 1) \varepsilon^2} \int_{\Omega} b \bigl((t u - q_\varepsilon)_+^{p + 1} - (u - q_\varepsilon)_+^{p + 1}\bigr)\\
& = \mathcal{E}_{\varepsilon} (tu) + \frac{1}{\varepsilon^2} \int_{\Omega} b \Bigl(\frac{(1- t^2)(u - q_\varepsilon)_+^{p}u}{2} + \frac{(t u - q_\varepsilon)_+^{p + 1} - (u - q_\varepsilon)_+^{p + 1}}{p + 1}\Bigr),
\end{split}
\]
from which one deduces since \(p \gt 1\) that \(\mathcal{E}_{\varepsilon} (t u) \gt \mathcal{E}_{\varepsilon} (u)\). This proves that 
\[
 \inf_{u \in H^1_0 (\Omega, b) \setminus \{0\}} \sup_{t \gt 0} \mathcal{E}_{\varepsilon} (t u) \lt \inf_{u \in \nehari} \mathcal{E}_{\varepsilon}(u).
\]
It is clear that 
\[
 \inf_{\gamma \in \Gamma_\varepsilon} \max_{t \in [0, 1]} \mathcal{E}_{\varepsilon} \bigl( \gamma (t)\bigr) \lt \inf_{u \in H^1_0 (\Omega, b) \setminus \{0\}} \sup_{t \gt 0} \mathcal{E}_{\varepsilon} (t u). 
\]
Let us now prove that 
\begin{equation}
\label{eqNehariPath}
   \inf_{u \in \nehari} \mathcal{E}_{\varepsilon}(u) \lt \inf_{\gamma \in \Gamma_\varepsilon} \max_{t \in [0, 1]} \mathcal{E}_{\varepsilon} \bigl( \gamma (t)\bigr).
\end{equation}
Let \(\gamma \in \Gamma_\varepsilon\) and define 
\(h \in C([0, 1]; \R)\) for \(t \in [0, 1]\) by \(h (t) = \dualprod{\mathcal{E}_{\varepsilon}' (\gamma(t))}{\gamma (t)}\). 
Since \(p \gt 1\), for every \(u \in H^1_0 (\Omega, b)\),
\[
  \int_{\Omega} b (u - q_\varepsilon)_+^p u 
\lt \int_{\Omega} b (u - q_\varepsilon)^{p - 1}_+ \bigl(u - \tfrac{q_\varepsilon}{2}\bigr)^2
\lt \int_{\Omega} b \bigl(u - \tfrac{q_\varepsilon}{2}\bigr)^{p + 1}_+,
\]
we have
\[
 \lim_{t \to 0} \frac{h (t)}{\displaystyle \int_{\Omega} \frac{\abs{\nabla \gamma (t)}^2}{b}} =  1,
\]
and thus \(h (t) > 0\) for \(t > 0\) close to \(0\).
On the other hand, by lemma~\ref{lemmeGradientEnergie}, since \(p \gt 1\),
\[
   \mathcal{E}_{\varepsilon} (u) \gt \frac{1}{p + 1} \dualprod{\mathcal{E}_{\varepsilon}' (u)}{u}.
\]
Hence, one has \(h (1) \lt (p + 1) \mathcal{E}_{\varepsilon} \bigl(\gamma (1)\bigr) < 0\). 
By the intermediate value theorem, there exists \(t_* \in [0, 1]\) such that \(h (t_*) = 0\) and thus \(\gamma (t_*) \in \nehari\). Therefore,
\[
 \inf_{u \in \nehari} \mathcal{E}_{\varepsilon}(u) \lt \mathcal{E}_{\varepsilon}\bigl(\gamma (t_*)\bigr) \lt \max_{t \in [0, 1]} \mathcal{E}_{\varepsilon} \bigl( \gamma (t)\bigr),
\]
and \eqref{eqNehariPath} follows.

The existence of the Palais-Smale sequence comes from a consequence of the quantitative deformation lemma \cite{Wi96}*{theorem~2.9}.
\end{proof}

\subsection{Existence in bounded domains}
In the case where \(\Omega\) and \(b\) are bounded, the existence of solutions to \eqref{problemP} is quite standard.

\begin{proposition}
\label{propositionExistenceBounded}
If \(\Omega \subset \R^2\) is bounded and \(b\) and \(b^{-1}\) are bounded, then for every \(\varepsilon \in (0, 1)\), there exists a weak solution \(u_{\varepsilon} \in H^1_0 (\R^2_+, b)\) of problem \eqref{problemP} such that \(\mathcal{E}_{\varepsilon} (u_{\varepsilon}) = c_\varepsilon\).
\end{proposition}
\begin{proof}[Sketch of the proof]
Define for \(x \in \Omega\) and \(s \in \R\)
\[
 f (x, s) = \frac{b (x)}{\varepsilon^2} \bigl(s - q_\varepsilon (x)\bigr)_+^p,
\]
and
\[
 F (x, s) = \int_0^s f (x, t)\du t = \frac{b (x) \bigl(s - q_\varepsilon (x)\bigr)_+^{p + 1}}{(p + 1)\varepsilon^2}.
\]
The function \(f\) is a Carath\'edory function and for every \(s \in \R\) and \(x \in \Omega\), since \(q \gt 0\),
\begin{gather*}
 \abs{f (x, s)} \lt \frac{\sup_{\Omega} b}{\varepsilon^2} \abs{s}^p, \\
 0 \lt (p + 1) F (x, s) \lt s f(x, s).
\end{gather*}
Hence, the problem has a weak solution by the mountain pass theorem \cite{Ra86}*{theorem~2.15}.
\end{proof}

\subsection{Existence in unbounded domains}
In unbounded domains, we prove the existence following the ideas of the concentration-compactness method of P.-L. Lions \citelist{\cite{Li84}\cite{Ra92}}. The existence will depend on the geometry of \(\Omega\), \(b\) and \(q\).

\subsubsection{Sobolev inequalities for truncated functions in unbounded domains}
In order to show that the functional \(\mathcal{E}_{\varepsilon}\) is well-defined on \(H^1_0 (\R^n, b)\) and admits critical points, we first study its nonlinear term.
We begin by proving a weighted Sobolev inequality.

\begin{lemma} 
\label{lemmaSobolevAvecPoids}
Let \(q \gt 2\), \(\alpha > - 1\) and \(\beta \in \R\).
If
\begin{equation*}
\frac{\beta - 2}{q} = \frac{\alpha}{2},
\end{equation*} 
then there exists \(C>0\) such that for every \(u \in H^1_0 (\R^2_+, x_1^{-\alpha})\), 
\begin{equation*}
\int_{\R^2_+} \frac{\abs{u(x)}^q}{x_1^{\beta}} \du x \lt C \Bigl(\int_{\R^2_+} \frac{\abs{\nabla u(x)}^2}{x_1^{\alpha}} \du x   \Bigr)^\frac{q}{2}.
\end{equation*}
\end{lemma}

This inequality should be known but we could not find it in the litterature. 
It is a limiting case of a known family of weighted Sobolev inequalities \cite{Ma11}*{\S 2.1.7}.

\begin{proof}[Proof of lemma~\ref{lemmaSobolevAvecPoids}] 
By the classical Sobolev inequality, there exists \(C > 0\) such that for every \(u \in H^1_0 (\R^2_+, x_1^{-\alpha})\),
\begin{equation*}
\int_{(1,2) \times \R} \frac{\abs{u(x)}^q}{x_1^{\beta}} \du x  
\lt C \Bigl(\int_{(1,2) \times \R} \frac{\abs{\nabla u(x)}^2}{x_1^{\alpha}} + \frac{\abs{u(x)}^2}{x_1^{\alpha + 2}} \du x \Bigr)^\frac{q}{2}. 
\end{equation*}
Since \(\frac{\beta - 2}{q} = \frac{\alpha}{2}\), the inequality is homogeneous, so that we have for every \(k \in \mathbb{Z}\),
\begin{equation*}
\int_{(2^k,2^{k+1}) \times \R} \frac{\abs{u(x)}^q}{x_1^\beta} \du x  \lt C \Bigl(\int_{(2^{k},2^{k+1}) \times \R} \frac{\abs{\nabla u(x)}^2}{x_1^{\alpha}} + \frac{\abs{u(x)}^2}{x_1^{\alpha + 2}} \du x \Bigr)^\frac{q}{2}.
\end{equation*}
Summing over \(k\), we obtain since \(q \gt 2\),
\begin{align*}
\int_{\R^2_+} \frac{\abs{u(x)}^q}{x_1^\beta} \du x & 
\lt C \sum_{k \in \mathbb{Z}}  \Bigl(\int_{(2^{k-1},2^{k+2}) \times \R} \frac{\abs{\nabla u(x)}^2}{x_1^{\alpha}} + \frac{\abs{u(x)}^2}{x_1^{\alpha + 2}} \du x  \Bigr)^\frac{q}{2} \\
& \lt C   \Bigl(\int_{\R^2_+} \frac{\abs{\nabla u(x)}^2}{x_1^{\alpha}} + \frac{\abs{u(x)}^2}{x_1^{\alpha + 2}} \du x \Bigr)^\frac{q}{2}.
\end{align*}
We conclude using the Hardy inequality that states that for \(\alpha  \ne - 1\), 
\[
 \int_{\R^2_+} \frac{\abs{u (x)}^2}{x_1^{\alpha + 2}}\du x
\lt \Bigl(\frac{2}{\alpha + 1}\Bigr)^2 \int_{\R^2_+} \frac{\abs{\nabla u (x)}^2}{x_1^{\alpha}} \du x. \qedhere
\]
\end{proof}

The crucial tool to show that the functional \(\mathcal{E}_{\varepsilon}\) is well-defined is a weighted Sobolev inequality for truncations.

\begin{lemma}
\label{lemmaSobolevTronque}
Let \(r \gt 0\), \(\alpha > -1\) and \(\beta \in \R\).
If 
\begin{align*}
  \beta & \lt (r - 1) (\alpha +1) + 1 &
 &\text{ and } &
  \beta  & \lt \frac{r \alpha}{2} + 2,
\end{align*}
then there exists \(C>0\) such that for all \(u \in H^1_0 (\R^2_+, x_1^{-\alpha})\),
\begin{equation*}
\int_{\R^2_+} \frac{\bigl( u(x) - W x_1^{\alpha +1} \bigr)^{r}_+}{x_1^\beta} \du x 
\lt \frac{C}{W^{\frac{r \alpha - 2 (\beta - 2)}{\alpha + 2}}}
 \Bigl(\int_{\R^2_+} \frac{\abs{\nabla u (x)}^2}{x_1^\alpha} \du x \Bigr)^{\frac{r (\alpha + 1) - (\beta - 2)}{\alpha + 2}}.
\end{equation*}
Moreover, the map
\[
 H^1_0 (\R^2_+, x_1^{-\alpha}) \to \R : u \mapsto \int_{\R^2_+} \frac{\bigl( u(x) - W x_1^{\alpha +1} \bigr)^{r}_+}{x_1^\beta} \du x
\]
is continuous.
\end{lemma}

Similar inequalities were proved by a variational argument and scaling for \(\alpha = 1\),\(\beta = - 1\) and \(r \gt 1\) \citelist{\cite{BeFr74}*{lemma IIIA}\cite{Ya95}*{lemma I.1 (1)}}. 
Similar inequalities were proved when \(\alpha = 1\)
 and \(\beta = 3\) and \(r = 0\) with an isometry with \(H^1 (\R^5)\) \cite{AmFr86}*{lemma~2.1} and when \(\alpha = \beta = 0\) with an isometry with \(H^1 (\R^4)\) \cite{Ya91}*{lemma~2.5}. (See \cite{We53} for a general explanation of those isometries.) In the latter case Smets and Van Schaftingen have given a proof of the inequality based directly on the classical Hardy and Sobolev inequalities \cite{SmVS10}*{proposition~4.2}.

\begin{proof}[Proof of lemma~\ref{lemmaSobolevTronque}]
For every \(q \gt r\) and for every \(x = (x_1, x_2) \in \R^2_+\),
\[
  \frac{\bigl(u (x) - W x_1^{\alpha + 1}\bigr)^r_+}{x_1^\beta}
 \lt \frac{\abs{u(x)}^q}{W^{q - r} x_1^{(q - r)(\alpha + 1) + \beta}}.
\]
Set now 
\[
 q = 2 \frac{r (\alpha + 1) - (\beta - 2)}{\alpha + 2}.
\]
After having observed that by our assumptions \(q \gt \max(2, r)\),
we conclude by applying lemma~\ref{lemmaSobolevAvecPoids}.
The continuity follows from the same bound and Lebesgue's dominated convergence theorem.
\end{proof}

We also want to have an inequality that relates the local behaviour with the global behaviour. 
Such results originate in the work of P.-L. Lions \cite{Li84}*{II, lemma~I.1} (see also \cite{Wi96}*{lemma~1.21}).

\begin{lemma}\label{ThReseauBoules}
If \(\alpha  > -1\) and \(r \gt 0\), 
then there exists \(C>0\) such that for all \(u \in H^1_0 (\R^2_+, x_1^{-\alpha})\) and \(W > 0\),
\begin{multline*}
\int_{\R^2_+}  \frac{\bigl( u(x) - W x_1^{\alpha +1} \bigr)^{r}_+}{x_1^{\beta}}  \du x \\
\lt \frac{C}{W^{\frac{r \alpha - 2 (\beta - 2)}{\alpha + 2}}} \biggl( \int_{\R^2_+} \frac{\abs{\nabla u (x)}^2}{x_1^\alpha} \du x + \frac{1}{W^{\frac{4}{\alpha + 2}}} \Bigl(  \int_{\R^2_+} \frac{\abs{\nabla u (x)}^2}{x_1^\alpha} \du x\Bigr)^{1 + \frac{2}{\alpha + 2}} \biggr)\\
\times \biggl( \sup_{a \in \R} \int_{ \R_+ \times (a - 1, a + 1)} \frac{\bigl( u(x) - W x_1^{\alpha +1} \bigr)^{r}_+}{x_1^{\beta}} \du x\biggr)^{1 - \frac{\alpha + 2}{r (\alpha + 1) - (\beta - 2)}}.
\end{multline*}
\end{lemma}
\begin{proof} 
Chosse \(\eta \in C^\infty (\R)\) such that \(\eta = 1\) on \([-1, 1]\) and \(\supp \eta \subset [-2, 2]\). 
For every \(a \in \R\) and \(x = (x_1, x_2) \in \R_+^2\), 
define \(\theta_a (x) = \eta (x_2 - a)\).
For every \(v \in H^1_0 (\R^2_+, x_1^{-\alpha})\), we have by lemma~\ref{lemmaSobolevTronque},
\begin{multline*}
  \int_{\R_+ \times (a - 1, a + 1)}
  \frac{\bigl(v (x) - \frac{W}{2} x_1^{\alpha + 1}\bigr)^r_+}{x_1^\beta} \du x\\
\shoveleft{ \lt 
  \int_{\R_+^2}
  \frac{\bigl(\theta_a (x) v (x) - \frac{W}{2} x_1^{\alpha + 1}\bigr)^r_+}{x_1^\beta} \du x}\\
 \lt 
\frac{C}{W^{\frac{r \alpha - 2 (\beta - 2)}{\alpha + 2}}}
 \Bigl(\int_{\R_+ \times (a - 2, a + 2)} \frac{\abs{\nabla v (x)}^2}{x_1^\alpha} + \frac{\abs{v (x)}^2}{x_1^\alpha} \du x \Bigr)^{\frac{r (\alpha + 1) - (\beta - 2)}{\alpha + 2}}.
\end{multline*}
This implies that 
\begin{multline*}
\int_{\R_+ \times (a - 1, a + 1)}
  \frac{\bigl(v (x) - \frac{W}{2} x_1^{\alpha + 1}\bigr)^r_+}{x_1^\beta} \du x\\
\lt C 
\frac{C}{W^{2 - \frac{\alpha + 2}{r (\alpha + 1) - (\beta - 2)} }}
 \Bigl(\int_{\R_+ \times (a - 2, a + 2)} \frac{\abs{\nabla v (x)}^2}{x_1^\alpha} + \frac{\abs{v (x)}^2}{x_1^\alpha} \du x \Bigr)\\
\times \Bigl(\int_{\R_+ \times (a - 1, a + 1)}
  \frac{\bigl(v (x) - \frac{W}{2} x_1^{\alpha + 1}\bigr)^r_+}{x_1^\beta} \du x
\Bigr)^{1 - \frac{\alpha + 2}{r (\alpha + 1) - (\beta - 2)}}.
\end{multline*}

For \(u \in H^1_0 (\R^2_+, x_1^{-\alpha})\), set now
\[
 v (x) = \bigl(u (x) - \tfrac{W}{2}x_1^{\alpha + 1}\bigr)_+.
\]
We apply the previous inequality, noting that by lemma~\ref{lemmaSobolevTronque} 
\[
\begin{split}
  \int_{\R^2_+} \frac{\abs{\nabla v (x)}^2}{x_1^\alpha}\du x
  & \lt 2 \int_{\R^2_+} \frac{\abs{\nabla u (x)}^2}{x_1^\alpha}\du x
+ 2 \int_{\R^2_+} x_1^\alpha (u (x) - W x_1^{\alpha +1})^{p+1}\du x\\
 & \lt C \int_{\R^2_+} \frac{\abs{\nabla u (x)}^2}{x_1^\alpha}\du x
\end{split}
\]
and 
\[
 \int_{\R^2_+} \frac{\abs{v (x)}^2}{x_1^\alpha}\du x
\lt \frac{C}{W^\frac{4}{\alpha + 2}} \Bigl(\int_{\R^2_+} \frac{\abs{\nabla u (x)}^2}{x_1^\alpha}\du x \Bigr)^{1+ \frac{2}{\alpha + 2}}.\qedhere
\]
\end{proof}

As a consequence of the previous lemmas, we have

\begin{lemma}
\label{lemmaEeWell}
Let \(\Omega \subset \R^2_+\), \(\alpha \gt 0\), \(b (x) = \frac{1}{x_1^\alpha}\) and \(q : \R^2 \to [0, \infty)\) be measurable.
If 
\[
  \inf_{x \in \Omega} \frac{q (x)}{x_1^{\alpha + 1}} > 0,
\]
then for every \(\varepsilon \in (0, 1)\), the functional \(\mathcal{E}_{\varepsilon}\) is well-defined and continuously Fr\'echet-differentiable.
Moreover
\[
 \lim_{u \to 0} \frac{\mathcal{E}_{\varepsilon} (u)}{\displaystyle \int_{\Omega} \frac{\abs{\nabla u}^2}{b}} > 0.
\]
and there exists a constant \(c > 0\) depending only on \(p\), \(\alpha\), \(\inf_{x \in \Omega} \frac{q (x)}{x_1^{\alpha + 1}}\) and \(\varepsilon\)  such that for every \(u \in H^1_0 (\R^2_+, b)\),
\[
 \max_{t > 0} \mathcal{E}_{\varepsilon} (t u) \gt c.
\]
\end{lemma}
\begin{proof}
The well-definitess, the smoothness and the asymptotic behaviour around \(0\) follow from lemma~\ref{lemmaSobolevTronque}.
By the same lemma, we have
\[
\begin{split}
  \mathcal{E}_{\varepsilon} (t v)
&\gt \frac{t^2}{2} \int_{\R^2_+} \frac{\abs{\nabla v}^2}{b} - \frac{1}{(p + 1) \varepsilon^2}
\int_{\R^2_+} b (t v - q_\varepsilon)^{p + 1}_+\\
&\gt \frac{t^2}{2} \int_{\R^2_+} \frac{\abs{\nabla v}^2}{b} - C   \Bigl( \int_{\R^2_+} t^{2} \frac{\abs{\nabla v}^2}{b}\Bigr)^{1 + (p + 1) \frac{\alpha + 1}{\alpha + 2}};
\end{split}
\]
by maximizing the right-hand side over \(t > 0\), we reach the conclusion.
\end{proof}

A more precise analysis shows that the conclusion of lemma~\ref{lemmaEeWell} still holds for \(\alpha \in (0, 1)\) under some additional restriction on \(p\).

\subsubsection{The translation-invariant case}
We now show that problem \eqref{problemP} has at least a nontrivial solution when for a translation invariant problem.
We say that a set \(\Omega \subset \R^2\) is translation-invariant, if for every \((x_1, x_2) \in \Omega\), \((x_1, x_2 + s) \in \Omega\) and that a function \(g : \Omega \to \R\) is translation-invariant if for every \((x_1,x_2) \in \Omega\) and \(s \in \R\),
\begin{equation*}
  g (x_1,x_2) = g(x_1,x_2+s).
\end{equation*}

\begin{proposition}
\label{propositionExistenceInvariant}
Let \(\alpha \gt 0\) and let \(\varepsilon \in (0, 1)\). If \(\Omega \subset \R_+^2\) is open and translation-invariant, if for every \(x \in \Omega\), \(b (x) = x_1^\alpha\), if \(q : \Omega \to \R\) is  measurable and translation-invariant and if 
\[
 \inf_{x \in \Omega} \frac{q (x)}{x_1^{\alpha + 1}} > 0,  
\]
then there exists a solution \(u_{\varepsilon} \in H^1_0 (\Omega, b)\) of problem \eqref{problemP} such that \(\mathcal{E}_{\varepsilon} (u_{\varepsilon}) = c_\varepsilon\).
\end{proposition}

When \(\Omega = \R_+^2\), the result is due to Ambrosetti and Yang for \(\alpha = 1\) \citelist{\cite{Ya91}*{theorem 1}\cite{AmYa90}*{theorem 1}} and to Yang for \(\alpha = 0\) \cite{Ya95}*{theorem 1}.

\begin{lemma}
\label{lemmaPalaisSmaleInvariant}
Let \(\alpha \gt 0\) and \(\varepsilon \in (0, 1)\).
If \(\Omega \subset \R^2_+\) is open and translation-invariant, if for every \(x \in \Omega\), \(b (x) = x_1^\alpha\), if \(q : \Omega \to \R\) and \(q^n : \Omega \to \R\) are measurable and translation-invariant and
if
\begin{align}
& \text{for every \(x \in \Omega\)} &
\tag{a} \label{eqLimqn} \lim_{n \to \infty} q^n (x) &= q (x),\\
\tag{b} & & \label{eqInfqn} \inf_{n \in \N} \inf_{x \in \Omega} \frac{q^n (x)}{x_1^{\alpha + 1}} &> 0,\\
\tag{c} & & \label{eqliminfEe} \liminf_{n \to \infty} \mathcal{E}_{\varepsilon}^n(u^n) & > 0,\\
\tag{d} & & \label{eqlimsupEe} \limsup_{n \to \infty} \mathcal{E}_{\varepsilon}^n(u^n) & < \infty,\\
\tag{e}& & \label{eqConvergencedEe}\mathcal{E}_{\varepsilon}^n{}'(u^n) & \goesto 0 \ & & \text{ in } \bigl(H^1_0 (\Omega, b)\bigr)' \text{ as } n \goesto \infty,
\end{align}
where \(\mathcal{E}_{\varepsilon}^n\) denotes the functional associated to \(q^n\),
then there exists \(u \in H^1_0 (\R^2, b)\) such that \(\mathcal{E}_{\varepsilon}'(u) = 0\) and 
\[
  \mathcal{E}_{\varepsilon} (u) \lt \liminf_{n \to \infty} \mathcal{E}_{\varepsilon}^n (u^n).
\]
\end{lemma}

In the proof of lemma~\ref{lemmaPalaisSmaleInvariant}, we follow the strategy of Rabinowitz \cite{Ra92}*{theorem 3.21}.

\begin{proof} 
By our assumption \eqref{eqConvergencedEe} and by lemma~\ref{lemmeGradientEnergie}, we have as \(n \to \infty\),
\[
 \Bigl(\frac{1}{2} - \frac{1}{p + 1} \Bigr) \int_{\Omega} \frac{\abs{\nabla u^n}^2}{b} 
\lt \mathcal{E}_{\varepsilon}^n (u^n) - \frac{1}{p + 1} \dualprod{\mathcal{E}_{\varepsilon}^n{}' (u^n)}{u^n} =  \mathcal{E}_{\varepsilon}^n (u^n)  + o (1) \Bigl(\int_{\Omega} \frac{\abs{\nabla u^n}^2}{b} \Bigr)^\frac{1}{2}.
\]
By the assumption \eqref{eqlimsupEe}, the sequence \((u^n)_{n \in \N}\) is thus bounded in \(H^1_0 (\Omega, b)\). 
Applying again \eqref{eqConvergencedEe}, we have, as \(n \to \infty\),
\begin{equation*}
\int_{\Omega} \frac{\abs{\nabla u^n}^2}{b} = \frac{1}{\varepsilon^2} \int_{\Omega} b (u^n - q_\varepsilon)^p_+ u^n + o(1). 
\end{equation*}
By \eqref{eqInfqn}, we have \(W = \inf_{n \in \N} \inf_{x \in \Omega} \frac{q^n (x)}{x_1^{\alpha + 1}} > 0\). 
Setting for \(x \in \Omega\), \(\underline{q}_\varepsilon (x) = (\log \tfrac{1}{\varepsilon}) \frac{W}{2} x_1^{\alpha +1}\), we have since \(p > 1\),
\[
\begin{split}
 \frac{1}{\varepsilon^2} \int_{\Omega} b (u^n - q_\varepsilon^n)^p_+ u^n
&\lt  \frac{1}{\varepsilon^2} \int_{\Omega} b (u^n - 2 \underline{q}_\varepsilon)^p_+ u^n\\
 &= \frac{1}{\varepsilon^2} \int_{\Omega} b \bigl(u^n - 2 \underline{q}_\varepsilon \bigr)^{p-1}_+ \bigl( (u^n - \underline{q}_\varepsilon)^2 - \underline{q}_\varepsilon^2 \bigr)  \\ 
  & \lt \frac{1}{\varepsilon^2} \int_{\Omega} b \bigl(u^n - \underline{q}_\varepsilon\bigr)^{p + 1}_+.
\end{split}
\]
On the other hand, by lemma~\ref{lemmaSobolevTronque}, there exists \(C > 0\) such that
\[
  \int_{\Omega} b \bigl(u^n - \underline{q}_\varepsilon\bigr)^{p + 1}_+ \lt C \Bigl(\int_{\Omega} \frac{\abs{\nabla u^n}^2}{b} \Bigr)^{1 + (p + 1)\frac{\alpha + 1}{\alpha + 2}}.
\]
Hence, since \(1 + (p + 1)\frac{\alpha + 1}{\alpha + 2} > 1\) and \eqref{eqliminfEe} holds, we deduce by lemma~\ref{lemmaEeWell} that 
\begin{equation*}
\liminf_{n \goesto \infty} \int_{\Omega} b \bigl(u^n - \underline{q}_\varepsilon\bigr)^{p + 1}_+ >0.
\end{equation*}
Since the sequence \((u^n)_{n \in \N}\) is bounded in \(H^1_0 (\Omega, b)\), this implies by lemma~\ref{ThReseauBoules} that
\[
 \liminf_{n \goesto \infty} \sup_{a \in \R} \int_{\Omega \cap (\R \times (a - 1, a + 1))} b \bigl(u^n - \underline{q}_\varepsilon\bigr)^{p + 1}_+ >0;
\]
hence there exists a sequence \((a^n)_{n \in \N}\) in \(\R\) such that 
\[
 \liminf_{n \goesto \infty} \int_{\Omega \cap (\R \times (a^n - 1, a^n + 1))} b \bigl(u^n - \underline{q}_\varepsilon\bigr)^{p + 1}_+ > 0.
\]
Define now for \(n \in \N\) and \(x = (x_1, x_2) \in \Omega\), \(v^n (x) = u^n (x_1, a^n + x_2)\). It is clear that \(v^n \in H^1_0 (\Omega, b)\),
\begin{equation*}
\mathcal{E}_{\varepsilon}^n(v^n) \goesto c_\varepsilon^\infty \quad \text{ and } \quad
\mathcal{E}_{\varepsilon}^n{}'(v^n) \goesto 0 \ \text{ in } \bigl(H^1_0 (\Omega, b)\bigr)' \text{ as } n \goesto \infty. 
\end{equation*}

Since the sequence \((v^n)_{n \in \N}\) is bounded in \(H^1_0 (\Omega, b)\), up to a subsequence, one can thus assume that \(v^n \weakto u\) weakly in \(H^1_0 (\Omega, b)\). 
By Rellich's compactness theorem, since \(\alpha \ge 0\),
\[
 \int_{\Omega \cap (\R  \times ( - 1, 1))} b \bigl(u - \underline{q}_\varepsilon\bigr)^{p + 1}_+
= \liminf_{n \goesto \infty} \int_{\Omega \cap (\R  \times (- 1, 1))} b \bigl(v^n - \underline{q}_\varepsilon\bigr)^{p + 1}_+ > 0,
\]
so that \(u \ne 0\).
By the weak convergence in \(H^1_0 (\Omega, b)\), the Rellich compactness theorem and by \eqref{eqLimqn} and \eqref{eqInfqn}, for every \(\varphi \in C^\infty_c (\Omega)\),
\[
\begin{split}
 0 & = \lim_{n \to \infty} \frac{1}{2} \int_{\Omega} \frac{\nabla v^n \cdot \nabla \varphi }{b} - \frac{1}{\varepsilon^2}  \int_{\Omega} b (v^n-q_\varepsilon^n)_+^p \varphi\\
 & = \frac{1}{2} \int_{\Omega} \frac{\nabla u \cdot \nabla \varphi }{b} - \frac{1}{\varepsilon^2}  \int_{\Omega} b (u-q_\varepsilon)_+^p \varphi.
\end{split}
\]
So, \(u\) is a weak solution of \eqref{problemP}  and \(u \in \nehari\). 

As \(u\) satisfies the Nehari constraint, by \eqref{eqLimqn} and by Fatou's lemma,
we can write 
\begin{align*}
\lim_{n \goesto \infty} \mathcal{E}_{\varepsilon}^n (u^n)
& = \lim_{n \goesto \infty} \frac{1}{\varepsilon^2}  \int_{\Omega} b (v^n - q_\varepsilon^n)_+^p u^n - \frac{1}{\varepsilon^2}  \int_{\Omega} b \frac{(v^n - q_\varepsilon^n)^{p + 1}_+}{p + 1}   \\ 
& \gt   \frac{1}{\varepsilon^2}  \int_{\Omega} b (u - q_\varepsilon)_+^p u - \frac{1}{\varepsilon^2}  \int_{\Omega} b \frac{(u - q_\varepsilon)^{p + 1}_+}{p + 1}   \\ 
&= \mathcal{E}_{\varepsilon}(u).\qedhere
 \end{align*}
\end{proof}

As a first application of lemma~\ref{lemmaPalaisSmaleInvariant}, we prove proposition~\ref{propositionExistenceInvariant}.

\begin{proof}[Proof of proposition~\ref{propositionExistenceInvariant}]
By lemma~\ref{lemmaCritical}, there exists a sequence Palais-Smale sequence \((u^n)_{n \in \N}\) associated to the critical level \(c_\varepsilon\), that is 
\begin{equation*}
\mathcal{E}_{\varepsilon}(u^n) \goesto c_\varepsilon \quad \text{ and } \quad
\mathcal{E}_{\varepsilon}'(u^n) \goesto 0 \ \text{ in } \bigl(H^1_0 (\Omega, b)\bigr)' \text{ as } n \goesto \infty. 
\end{equation*}
By lemma~\ref{lemmaPalaisSmaleInvariant} with \(\mathcal{E}_{\varepsilon}^n = \mathcal{E}_{\varepsilon}\), there exists \(u \in H^1_0 (\R^2_+, b) \setminus \{0\}\) such that \(\mathcal{E}_{\varepsilon}' (u) = 0\) and \(\mathcal{E}_{\varepsilon} (u) \lt c_\varepsilon\). 
Since \(u \ne 0\) and \(\mathcal{E}_{\varepsilon}' (u) = 0\), we have \(u \in \nehari\) and thus \(\mathcal{E}_{\varepsilon} (u) \gt c_\varepsilon\).
\end{proof}

We shall also need to know that \(c_\varepsilon\) depends continuously on \(q_\varepsilon\).

\begin{lemma}
\label{lemmaInvariantContinuous}
Let \(\alpha \gt 0\) and \(\varepsilon \in (0, 1)\).
If \(\Omega \subset \R^2_+\) is open and translation-invariant, if for every \(x \in \Omega\), \(b (x) = x_1^\alpha\), if \(q : \Omega \to \R\) and \(q^n : \Omega \to \R\) are measurable and translation-invariant and
if
\begin{align*}
& \text{for every \(x \in \Omega\)} &
\lim_{n \to \infty} q^n (x) &= q (x),
\end{align*}
and
\[
 \inf_{n \in \N} \inf_{x \in \Omega} \frac{q^n (x)}{x_1^{\alpha + 1}} > 0,
\]
then 
\[
 \lim_{n \to \infty} c_\varepsilon^n = c_\varepsilon.
\]
where \(c_\varepsilon^n\) denotes the critical level of the functional associated to \(q^n\).
\end{lemma}
\begin{proof}
By proposition~\ref{propositionExistenceInvariant}, there exists \(u \in H^1_0 (\Omega, b)\) such that \(\mathcal{E}_{\varepsilon} (u) = c_\varepsilon\) and \(\mathcal{E}_{\varepsilon}' (u) = 0\).
Choose \(t_n > 0\) such that 
\[
  \max_{t > 0} \mathcal{E}_{\varepsilon}^n (tu) = \mathcal{E}_{\varepsilon}^n (t_n u).
\]
One has \(\lim_{n \to \infty} t_n = 1\) and thus 
\[
 \limsup_{n \to \infty} c_\varepsilon^n \lt \lim_{n \to \infty} \mathcal{E}_{\varepsilon}^n (t_n u) = \mathcal{E}_{\varepsilon} (tu) = c_\varepsilon.
\]

On the other hand, by lemma~\ref{lemmaCritical} and a diagonal argument, there exists a sequence \((u^n)_{n \in \N}\) in \(H^1_0 (\Omega, b)\) such that
\begin{equation*}
\mathcal{E}_{\varepsilon}^n(u^n) - c_\varepsilon^n \to 0 \quad \text{ and } \quad
\mathcal{E}_{\varepsilon}^n{}'(u^n) \goesto 0 \ \text{ in } \bigl(H^1_0 (\Omega, b)\bigr)' \text{ as } n \goesto \infty. 
\end{equation*}
By lemma~\ref{lemmaPalaisSmaleInvariant}, there exists \(u \in H^1_0 (\Omega, b) \setminus \{0\}\)  such that 
and \(\mathcal{E}_{\varepsilon}' (u) = 0\),
\[
 \liminf_{n \to \infty} c_\varepsilon^n = \liminf_{n \to \infty} \mathcal{E}_{\varepsilon}^n(u^n) \gt \mathcal{E}_{\varepsilon} (u).
\]
Since \(\mathcal{E}_{\varepsilon}'(u) = 0\) we have 
\[
 \mathcal{E}_{\varepsilon} (u) \gt c_\varepsilon.\qedhere
\]
\end{proof}

\subsubsection{Existence by strict inequalities}\label{SePerturbedProblem}
We turn now to the study of the problem in an unbounded subset of \(\R^2_+\) that needs not to be invariant under translations.

\begin{proposition}
\label{propositionStrict}
Let \(\Omega \subset \R^2_+\) be open and translation-invariant, \(\alpha \gt 0\) and \(\varepsilon \in (0, 1)\). 
Assume that for every \(x \in \Omega\), \(b (x) = x_1^\alpha\) and if \(q \in \Omega \to \) and \(\varepsilon > 0\),
\[
 \inf_{x \in \Omega} \frac{q (x)}{x_1^{\alpha + 1}} > 0,
\]
and that  
\[
  \liminf_{\abs{x} \to \infty} \frac{q (x)}{q^\infty (x)} \gt 1,
\]
where \(q^\infty : \Omega \to \R\) is measurable and translation-invariant and \(\inf_{x \in \Omega} \frac{q^{\infty}}{x_1^\alpha} > 0\).
If
\begin{equation*}
c_\varepsilon < c_\varepsilon^\infty,
\end{equation*}
where \(c_\varepsilon^\infty\) is the critical level defined by the functional associated to \(q^\infty\),
then there exists a solution \(u_{\varepsilon} \in H^1_0 (\Omega, b)\) of \eqref{problemP} such that \(\mathcal{E}_{\varepsilon}(u_{\varepsilon}) = c_\varepsilon\).
\end{proposition}

This kind of results goes back to the concentration-compactness method of P.-L. Lions \cite{Li84}. The presentation and the proof that we are giving are inspired by Rabinowitz \cite{Ra92} (see also \cite{SmVS10}).

\begin{proof}[Proof of proposition~\ref{propositionStrict}]
By lemma~\ref{lemmaCritical}, there exists a Palais-Smale sequence \((u^n)_{n \in \N}\) at level \(c_\varepsilon\). 
As in the proof of proposition~\ref{propositionExistenceInvariant}, by lemma~\ref{lemmeGradientEnergie}, the sequence is bounded in \(H^1_0 (\Omega, b)\) and we can thus assume without loss of generality that \(u^n \weakto u\) in \(H^1_0 (\Omega, b)\) as \(n \to \infty\). One has by Rellich's theorem for every \(\varphi \in C^\infty_c (\Omega)\)
\[
\begin{split}
 \frac{1}{2} \int_{\Omega} \frac{\nabla u \cdot \nabla \varphi }{b} - \frac{1}{\varepsilon^2}  \int_{\Omega} b (u-q_\varepsilon)_+^p \varphi & = \lim_{n \to \infty} \frac{1}{2} \int_{\Omega} \frac{\nabla u^n \cdot \nabla \varphi }{b} - \frac{1}{\varepsilon^2}  \int_{\Omega} b (u^n-q_\varepsilon)_+^p \varphi\\
 & = 0,
\end{split}
\]
so that \(u\) solves \eqref{problemP}. 

If \(u \ne 0\), then \(u \in \nehari\) and \(\mathcal{E}_{\varepsilon}(u) \geq c_\varepsilon\). 
Moreover, by Fatou's lemma,
\begin{align*}
 \mathcal{E}_{\varepsilon} (u) & = \frac{1}{\varepsilon^2}\int_{\Omega}  \frac{1}{2} (u-q_\varepsilon)^p_+ u - \frac{1}{p+1} (u-q_\varepsilon)^{p+1}_+  \\ 
 & \lt \liminf_{n \goesto \infty} \frac{1}{\varepsilon^2}\int_{\Omega} \frac{1}{2} (u^n-q_\varepsilon)^p_+ u^n - \frac{1}{p+1} (u^n-q_\varepsilon)^{p+1}_+  \\
 & = c_\varepsilon.
 \end{align*}
Hence we have \(\mathcal{E}_{\varepsilon} (u) = c_\varepsilon\) and the result follows.

If \(u = 0\) on \(\Omega\), for every \(\delta > 0\), define the energy functional \(\mathcal{E}^\delta_{\varepsilon}\) on \(H^1_0 (\R^2_+, b)\) by
\begin{equation*}
\mathcal{E}^\delta_{\varepsilon} (v) = \frac{1}{2} \int_{\Omega} \frac{\abs{\nabla v}^2}{b} - \frac{1}{(p+1) \varepsilon^2} \int_{\Omega} b (v-(1 - \delta) q_\varepsilon^\infty)^{p+1}_+,
\end{equation*}
where \(q_\varepsilon^\infty = \log \tfrac{1}{\varepsilon} q^\infty\) and the corresponding critical level
\begin{equation*}
c^{\delta}_\varepsilon = \inf_{v \in H^1_0 (\R^2_+, b) \setminus \{0\}} \sup_{t \gt 0} \mathcal{E}^\delta_{\varepsilon} (t v).
\end{equation*}
Choose now \(\tau_n\) such that \(\max_{\tau>0} \mathcal{E}^\delta_{\varepsilon}(\tau u^n) = \mathcal{E}^\delta_{\varepsilon}(\tau_n u^n)  \). We claim that the sequence \((\tau_n)_{n \in \N}\) is bounded. 
One has
\begin{equation*}
\begin{split}
  (\tau_n)^2  \int_{\R^2_+} \frac{\abs{\nabla u^n}^2}{b} & =  \frac{1}{\varepsilon^2} \int_{\R^2_+} b (\tau_nu^n-(1 - \delta) q_\varepsilon^\infty)^{p}_+ \tau_nu^n \\ 
  & \geq \max(\tau_n, 1)^{p+1} \frac{1}{\varepsilon^2} \int_{\R^2_+} b (u^n-(1 - \delta) q_\varepsilon^\infty)^{p}_+u^n.
\end{split}
\end{equation*}
Choosing \(R > 0\) such that \(q \gt (1 - \delta) q^\infty\) in \(\Omega \setminus B (0, R)\), note that by Rellich's compactness theorem, since \(\alpha \ge 0\),
\begin{equation*}
\begin{split}
  \liminf_{n \to \infty} \frac{1}{\varepsilon^2} \int_{\Omega} b (u^n-(1 - \delta) q_\varepsilon^\infty)^{p}_+ u^n
 & \gt \liminf_{n \goesto \infty} \frac{1}{\varepsilon^2} \int_{\Omega \setminus B (0, R)} b (u^n- q_\varepsilon)^{p}_+ u^n\\
&\gt \liminf_{n \goesto \infty} \frac{1}{\varepsilon^2} \int_{\Omega} b (u^n-q_\varepsilon)^{p}_+ u^n,
\end{split}
\end{equation*}
and that 
\[
\frac{1}{\varepsilon^2} \int_{\Omega} b (u^n- q_\varepsilon)^{p}_+ u^n 
\gt 2 \mathcal{E}_{\varepsilon} (u^n) - \dualprod{ \mathcal{E}_{\varepsilon}' (u^n)}{ u^n },
\]
therefore,
\[
 \liminf_{n \to \infty} \frac{1}{\varepsilon^2} \int_{\R^2_+} b (u_n - (1 - \delta) q_\varepsilon^\infty)^{p}_+u^n
 \gt 2 c_\varepsilon > 0,
\]
so that the sequence \((\tau_n)_{n \in \N}\) is bounded.

We compute
\begin{equation*}
\begin{split}
\mathcal{E}_{\varepsilon} (\tau_n u^n)  
& = \mathcal{E}^\delta_{\varepsilon} (\tau_nu^n) +\frac{1}{(p + 1) \varepsilon^2}\int_{\Omega} b (\tau_nu^n - (1 - \delta)q_\varepsilon^\infty)^{p + 1}_+ - b (\tau_n u^n - q_\varepsilon)^{p + 1}_+.
\end{split}
\end{equation*}
Choosing \(R\) as previously,
\[
  \int_{\Omega \setminus B (0, R)} b (\tau_nu^n - (1 - \delta) q_\varepsilon^\infty)^{p + 1}_+ - b (\tau_n u^n - q_\varepsilon)^{p + 1}_+ \gt 0
\]
and by Rellich's theorem, since \(\alpha \ge 0\) and the sequence \((\tau_n)_{n \in \N}\) is bounded
\[
 \lim_{n \to \infty} \int_{\Omega \cap B (0, R)} b (\tau_nu^n - (1 - \delta) q_\varepsilon^\infty)^{p + 1}_+ - b (\tau_n u^n - q_\varepsilon)^{p + 1}_+ = 0.
\]
We have thus 
\[
 \lim_{n \to \infty} \mathcal{E}_{\varepsilon} (\tau_n u^n) \gt \limsup_{n \to \infty} \mathcal{E}^\delta_{\varepsilon} (\tau_nu^n)
\]
and because \((u^n)_{n \in \N}\) is a Palais-Smale sequence we conclude that 
\begin{equation*}
c_\varepsilon \gt c^{\delta}_\varepsilon.
\end{equation*}
Since by lemma~\ref{lemmaInvariantContinuous},
\(
 \lim_{\delta \to 0} c^{\delta}_\varepsilon = c^\infty_\varepsilon,
\)
we conclude that
\begin{equation*}
c_\varepsilon \geq c_\varepsilon^\infty,
\end{equation*}
a contradiction with the assumed strict inequality.
\end{proof}

\section{Asymptotics of solutions}
\label{sectionAsymptotics}

In this section we study the asymptotics of solutions to \eqref{problemP}. 
We make the following assumptions on \(\Omega\), \(b\) and \(q\):
\begin{enumerate}[$(\mathcal{A}_1)$]
 \item for every \(\eta > 0\), there exists \(\delta  > 0\) such if \(x, y \in \Omega\) and \(\abs{x - y} \lt \delta \dist (x, \partial \Omega)\), then 
\[
 \Bigabs{\log \frac{b (x)}{b (y)}} \lt \eta,
\]
and
\[
 \Bigabs{\log \frac{q (x)}{q (y)}} \lt \eta,
\]  

\item there exists \(C \in \R\) such that for every \(x \in \Omega\),
\[
  \log \Bigl(1 + \dfrac{2 \dist (x, \partial \Omega) b (x)^{(p + 1)/2}}{q (x)^{(p - 1)/2}}\Bigr)
\lt C \frac{q (x)^2}{b(x)} ,
\]
\item \label{eqq} \(q \in H^1_\mathrm{loc} (\Omega)\), \(\inf_{\Omega} q > 0\) and 
\begin{equation*}
 - \dive \frac{\nabla q}{b} = 0
\end{equation*}
weakly in \(\Omega\),
\item the set \(\R^2 \setminus \Omega\) is unbounded and connected,
\item the functional \(\mathcal{E}_{\varepsilon}\) is well-defined and differentiable on \(H^1_0 (\Omega, b)\).
\end{enumerate}

The assumption \((\mathcal{A}_1)\) is equivalent with the uniform continuity with respect to the distance-ratio metric on \(\Omega\) of \(\log b\) and \(\log q\). 
When \(\Omega\) is a uniform domain, this is equivalent with the uniform continuity with respect to the quasi-hyperbolic metric on \(\Omega\). Those metrics are equivalent to the Poincar\'e metric on the ball and on the half-plane \citelist{\cite{Li07}\cite{GeOs79}\cite{GePa76}}.
Assumption \((\mathcal{A}_5)\) is satisfied under the assumptions of proposition~\ref{propositionExistenceBounded} or of lemma~\ref{lemmaEeWell}.

An important consequence of \((\mathcal{A}_3)\) is the following identity:

\begin{lemma}
\label{lemmaChangeWeight}
For every \(u \in H^1_0 (\Omega, b)\),
\[
 \int_{\Omega} \frac{\abs{\nabla u}^2}{b}
= \int_{\Omega} \frac{q^2}{b}\Bigabs{\nabla \Bigl(\frac{u}{q}\Bigr)}^2.
\]
\end{lemma}
\begin{proof}
 Take \(\frac{u^2}{q}\) as a test function in \eqref{problemP} and observe that 
\[
 2\nabla q \cdot \nabla \Bigl( \frac{u^2}{q}\Bigr)
=
 \abs{\nabla u}^2 - q^2 \Bigabs{\nabla \Bigl(\frac{u}{q} \Bigr)}^2.\qedhere
\]
\end{proof}

\subsection{Upper bound on the energy}\label{SeUpperBound}

As a first step, we prove an upper bound on \(c_\varepsilon\).
\begin{proposition} \label{ThUpperBound}
One has 
\begin{equation*}
\limsup_{\varepsilon \to 0} \frac{c_{\varepsilon}}{\log \frac{1}{\varepsilon}} \lt \pi \inf_\Omega \frac{q^2}{b}.
\end{equation*}
\end{proposition}

\begin{proof}
Choose \(U \in C^\infty(\R^2)\) such that \(U (x)=\log \frac{1}{\abs{x}}\) if \(\abs{x}\gt 1\) and \(U (x) > 0\) if \(\abs{x} < 1\), choose \(\rho > 0\) such that \(B (\Hat{x}, 2 \rho) \subset \Omega\) and choose a cut-off function \(\varphi \in C^\infty_c(B (0, 2 \rho))\) such that \(\varphi = 1\) in \(B(\Hat{x}, \rho)\).
Consider, for \(\tau \in \R\), the function \(v_{\varepsilon}^\tau \in C^\infty_c (\Omega)\) defined for \(x \in \Omega\) by
\begin{equation*}
v_{\varepsilon}^\tau (x) = q(x) \bigl( U \bigl( \tfrac{x-\Hat{x}}{\varepsilon} \bigr) + \log \tfrac{\tau}{\varepsilon} \bigr)  \varphi \bigl(\tfrac{x - \Hat{x}}{\rho}\bigr)
\end{equation*}
and define the function \(g_\varepsilon : \R \to \R\) for \(t \in \R\) by
\begin{equation*}
g_{\varepsilon} (\tau) = \frac{1}{\log \tfrac{1}{\varepsilon}} \dualprod{ \mathcal{E}_{\varepsilon}' (v_{\varepsilon}^\tau)}{ v_{\varepsilon}^\tau } = \frac{1}{\log \tfrac{1}{\varepsilon}} \Bigl( \int_{\Omega}  \frac{\abs{\nabla v_{\varepsilon}^\tau}^2}{b} - \frac{1}{\varepsilon^2} \int_{\Omega}  \ b (v_{\varepsilon}^\tau -q_\varepsilon)_+^p v_{\varepsilon}^\tau\Bigr).
\end{equation*}
We are going to show that for every \(\varepsilon\) small enough, there exists \(\tau_\varepsilon\) such that \(g^{\varepsilon}(\tau_\varepsilon) = 0\).

By lemma~\ref{lemmaChangeWeight}, we have
\begin{equation}
\label{eqQuadraticRewrite}
\int_{\Omega}  \frac{\abs{\nabla v_{\varepsilon}^\tau}^2}{b}
= \int_{\Omega}  \frac{q^2}{b}\Bigabs{\nabla \Bigl( \Bigl(\frac{v_{\varepsilon}^\tau}{q}\Bigr) \Bigr)}^2.
\end{equation}
First one observes that there exists \(C > 0\) such that for every \(\tau > 0\)
\begin{equation}
\label{eqQuadraticOuter}
 \int_{B (\Hat{x}, 2\rho) \setminus B (\Hat{x},\rho)}  \frac{q^2}{b}\Bigabs{\nabla \Bigl(\frac{v_{\varepsilon}^\tau}{q}\Bigr)}^2
 \lt C \bigl( 1 + \bigabs{\log \tfrac{\tau}{\rho}} \bigr)
\end{equation}
and that if \(\varepsilon \lt \rho\),
\[
 \int_{B (\Hat{x}, \varepsilon)} \frac{q^2}{b}\Bigabs{\nabla \Bigl(\frac{v_{\varepsilon}^\tau}{q}\Bigr)}^2
= \int_{B (0, 1)} \frac{q (\Hat{x} + \varepsilon y)^2}{b (\Hat{x} + \varepsilon y)} \abs{\nabla U (y)}^2 \du y,
\]
and thus 
\begin{equation}
\label{eqQuadraticInner}
 \lim_{\varepsilon \to 0} \int_{B (\Hat{x}, \varepsilon)} \frac{q^2}{b}\Bigabs{\nabla \Bigl(\frac{v_{\varepsilon}^\tau}{q}\Bigr)}^2 = \frac{q (\Hat{x})^2}{b (\Hat{x})}\int_{B (0, 1)} \abs{\nabla U}^2,
\end{equation}
uniformly in \(\tau > 0\).

Finally, since \(U (x) = \log \frac{1}{\abs{x}}\) if \(\abs{x} \gt 1\), we have if \(\varepsilon \lt \delta \lt \rho\),
\[
\begin{split}
 \biggabs{\frac{q (\Hat{x})^2}{b (\Hat{x})} 2 \pi \log \frac{\rho}{\varepsilon} - \int_{B (\Hat{x},\rho) \setminus B (\Hat{x}, \varepsilon)}\frac{q^2}{b}\Bigabs{\nabla \Bigl(\frac{v_{\varepsilon}^\tau}{q}\Bigr)}^2}
& \lt \int_{B (\Hat{x}, \rho)} \Bigabs{\frac{q (\Hat{x})^2}{b (\Hat{x})}- \frac{q (x)^2}{b(x)} } \frac{1}{\abs{x - \Hat{x}}^2} \du x\\
& \lt 2 \pi \Bigl( \omega (\rho) \log \frac{\rho}{\delta} + \omega (\delta) \log \frac{\varepsilon}{\delta} \Bigr),
\end{split}
\]
where
\[
 \omega (\delta) = \sup_{x \in B (\Hat{x}, \delta)} \Bigabs{\frac{q (\Hat{x})^2}{b (\Hat{x})}- \frac{q (x)^2}{b (x)}}.
\]
We have thus for every \(\delta > 0\),
\[
  \limsup_{\varepsilon \to 0}\, \Bigabs{2 \pi \frac{q (\Hat{x})^2}{b (\Hat{x})} - \frac{1}{\log \tfrac{1}{\varepsilon}} \int_{B (\Hat{x},\rho) \setminus B (\Hat{x}, \varepsilon)}\frac{q^2}{b}\Bigabs{\nabla \Bigl(\frac{v_{\varepsilon}^\tau}{q}\Bigr)}^2} \lt 2 \pi \omega (\delta).
\]
By continuity of \(q\) and \(b\), \(\lim_{\delta \to 0} \omega (\delta) = 0\) and thus we have proved
\begin{equation}
\label{eqQuadraticMiddle}
 \lim_{\varepsilon \to 0} \frac{1}{\log \tfrac{1}{\varepsilon}} \int_{B (\Hat{x},\rho) \setminus B (\Hat{x}, \varepsilon)}\frac{\abs{\nabla v_{\varepsilon}^\tau}^2}{b} = 2 \pi \frac{q (\Hat{x})^2}{b (\Hat{x})},
\end{equation}
uniformly in \(\tau > 0\).
Gathering \eqref{eqQuadraticRewrite}, \eqref{eqQuadraticOuter}, \eqref{eqQuadraticInner} and \eqref{eqQuadraticMiddle}, we have proved that 
\begin{equation}
\label{limQuadratic}
 \lim_{\varepsilon \to 0} \frac{1}{\log \tfrac{1}{\varepsilon}} \int_{\Omega}\frac{\abs{\nabla v_{\varepsilon}^\tau}^2}{b}
= 2 \pi \frac{q (\Hat{x})^2}{b (\Hat{x})},
\end{equation}
uniformly in \(\tau > 0\) in compact subsets. 

Now note that 
\begin{equation}
\label{decompNonlinear}
 \frac{1}{\varepsilon^2} \int_{\Omega} b (v_{\varepsilon}^\tau  - q_\varepsilon)_+^p v_{\varepsilon}^\tau
= 
 \frac{1}{\varepsilon^2} \int_{\Omega} b (v_{\varepsilon}^\tau  - q_\varepsilon)_+^p q_\varepsilon^\tau
 + \frac{1}{\varepsilon^2} \int_{\Omega} b (v_{\varepsilon}^\tau  - q_\varepsilon)_+^{p + 1}.
\end{equation}
If \(\varepsilon \tau \lt \rho\), one has for every \(x \in \Omega\),
\[
 (v_{\varepsilon}^\tau (x) - q_\varepsilon (x))_+= \bigl(U \bigl( \tfrac{x-\Hat{x}}{\varepsilon} \bigr) + \log \tau\bigr)_+.
\]
Hence we have since \(b\) and \(q\) are continuous
\begin{equation}
\label{limNonlinear}
\begin{split}
 \lim_{\varepsilon \to 0} \frac{1}{\varepsilon^2}
\int_{\Omega} b (v_{\varepsilon}^\tau  - q_\varepsilon)_+^{p + 1}
& = \lim_{\varepsilon \to 0} \int_{B (0, \tau)} b (\Hat{x} + y)^{p + 1} q (\Hat{x} + y) \bigl(U (y) + \log \tau \bigr)_+^{p + 1}\du y\\
& = b (\Hat{x}) q (\Hat{x})^{p + 1} \int_{B (0, \tau)}  \bigl(U + \log \tau\bigr)_+^{p + 1}
\end{split}
\end{equation}
and similarly
\begin{equation}
\label{limSemiNonlinear}
\lim_{\varepsilon \to 0} \frac{1}{\log \tfrac{1}{\varepsilon} \varepsilon^2} \int_{\Omega} b (v_{\varepsilon}^\tau  - q_\varepsilon)_+^{p} q_\varepsilon
= b (\Hat{x}) q (\Hat{x})^{p + 1} \int_{B (0, \tau)}  \bigl(U + \log \tau\bigr)_+^{p};
\end{equation}
the convergences are uniform on compact subsets.

By \eqref{decompNonlinear}, \eqref{limNonlinear} and \eqref{limSemiNonlinear}, we have thus proved that for every \(\tau > 0\), \(\lim_{\varepsilon \to 0} g_\varepsilon (\tau) = g (\tau)\),
where 
\[
 g(\tau) = \frac{2\pi q(\Hat{x})^2}{b(\Hat{x})}   - b(\Hat{x}) q(\Hat{x})^{p + 1}  \int_{\R^2} \bigl(U + \log \tau\bigr)_+^p. 
\]
Choose now \(\underline{\tau} > 0\) and \(\Bar{\tau} > 0\) such that \(g (\underline{\tau}) > 0\) and \(g (\Bar{\tau}) > 0\). Then, for \(\varepsilon>0\) sufficiently small, 
\(g_\varepsilon (\underline{\tau}) < 0 < g_\varepsilon(\Bar{\tau})\) and there exists a \(\tau_{\varepsilon} \in (\underline{\tau}, \Bar{\tau})\) such that \(g_{\varepsilon}(\tau_{\varepsilon}) =0\). 

One has then \(v_{\varepsilon}^{t_\varepsilon} \in \nehari\)
We can now compute the energy of \(v_{\varepsilon}^{\tau_{\varepsilon}}\) with the help of \eqref{limQuadratic} and \eqref{limNonlinear}, keeping in mind that the limits are uniform on compact subsets and that the family \((\abs{\log \tau_\varepsilon})_{\varepsilon >0}\) is bounded:
\begin{align*}
 \lim_{\varepsilon \to 0} \frac{1}{\log \tfrac{1}{\varepsilon}} \mathcal{E}_{\varepsilon} (v_{\varepsilon}^{\tau_\varepsilon}) &= \lim_{\varepsilon \to 0} \frac{1}{\log \tfrac{1}{\varepsilon}} \frac{1}{2} \int_{\Omega} \frac{\abs{\nabla v_{\varepsilon}^{\tau_\varepsilon}}^2}{b} - \lim_{\varepsilon \to 0} \frac{1}{\log \tfrac{1}{\varepsilon}}\frac{1}{\varepsilon^2} \int_{\Omega} b \frac{(v_{\varepsilon}^{\tau_\varepsilon}- q_\varepsilon)^{p + 1}_+}{p + 1} \\
 & = \pi \frac{q(\Hat{x})^2}{b(\Hat{x})}.
\end{align*}
The result follows by taking the infimum over \(\Hat{x} \in \Omega\).
\end{proof}

\begin{proposition}
\label{propImprovedUpper}
Under the assumption of the previous proposition, if there exists \(\Hat{x} \in \Omega\) such that 
\[
  \frac{q (\Hat{x})^2}{b (\Hat{x})} = \inf_{\Omega} \frac{q^2}{b} 
\]
and \(\frac{q^2}{b}\) is Dini-continuous in a neighbourhood of \(\Hat{x}\), then
\[
 c_{\varepsilon}(\Omega) \lt \pi  \log \frac{1}{\varepsilon} \inf_\Omega \frac{q^2}{b} + O (1)
\]
as \(\varepsilon \to 0\).
\end{proposition} 

Recall that \(f : \Omega \to \R\) is Dini-continuous in a neighbourhood of \(\Hat{x}\) if there exists \(\delta > 0\) and a nondecreasing function \(\omega : [0, \delta) \to \R\) such that 
\[
  \int_0^\delta \frac{\omega (t)}{t} \du t < \infty 
\]
and for every \(x,y \in B (\Hat{x}, \delta)\),
\[
  \abs{f (x) - f(y)} \lt \omega (\abs{x - y}).
\]

Remark that in order to have the improved bound the infimum should be achieved \emph{in the interior} of \(\Omega\) and \(\frac{q^2}{b}\) should satisfy some \emph{improved continuity} assumption at the minimum point. This is the case if \(\frac{q^2}{b}\) is coercive and continuously differentiable. Also note that by the classical regularity theory of De Giorgi \citelist{\cite{DG57}\cite{GiTr01}*{Chapter 8}}, since \(b\) is locally bounded and bounded away from \(0\), \(q\) is locally Dini-continuous. The condition is thus that \(b\) should be locally Dini-continuous.

\begin{proof}[Sketch of the proof of proposition~\ref{propImprovedUpper}]
The proof goes as the proof of proposition~\ref{ThUpperBound}, except that when studying \(\mathcal{E}_{\varepsilon} (v_{\varepsilon}^{\tau_\varepsilon})\), we note that our assumption allows us, by estimating \eqref{eqQuadraticMiddle}, to obtain
\[
  \lim_{\varepsilon \to 0} \frac{1}{\log \tfrac{1}{\varepsilon}} \int_{\Omega}\frac{\abs{\nabla v_{\varepsilon}^\tau}^2}{b}
= 2 \pi \log \tfrac{1}{\varepsilon} \frac{q (\Hat{x})^2}{b (\Hat{x})}  + O (1),
\]
as \(\varepsilon \to 0\), uniformly in \(\tau > 0\) over compact sets
instead of \eqref{limQuadratic}.
\end{proof}

\subsection{Asymptotic behaviour and lower bound on the energy}\label{SeAsymptoticVortexCore}

We are now going to study the asymptotics of a family of groundstates.
Thus, we assume that for every \(\varepsilon >0\), problem \eqref{problemP} possesses a nontrivial solution \(u_{\varepsilon} \in H^1_0 (\Omega, b)\) such that \(\mathcal{E}_{\varepsilon} (u_{\varepsilon}) = c_{\varepsilon}\). 

We define the vortex core to be the set
\[
A_{\varepsilon} = \bigl\{ x \in \Omega : u_{\varepsilon}(x) > q_\varepsilon (x) \bigr\}.
\]
Note that as \(u_{\varepsilon}\) is continuous inside \(\Omega\) by classical regularity theory \cite{GiTr01}*{theorem 8.22},  \(A_{\varepsilon}\) is an open subset of \(\Omega\).

We first give some integral identities involving the vortex core:

\begin{lemma}
\label{lemmaIntegralIdentity}
If \(u_{\varepsilon}\) is a solution of \eqref{problemP} then 
\begin{gather}
\tag{a}  \label{eqIntIDExterior}
 \frac{1}{\varepsilon^2} \int_{\Omega} (u_{\varepsilon} - q_\varepsilon)_+^p q_\varepsilon = \int_{\Omega} \frac{\abs{\nabla u_{\varepsilon}}^2}{b}
- \int_{A_{\varepsilon}} \frac{\abs{\nabla (u_{\varepsilon} - q_\varepsilon)}^2}{b}\\
\tag{b}  \label{eqIntIDInterior}\frac{1}{\varepsilon^2} \int_{A_{\varepsilon}} (u_{\varepsilon} - q_\varepsilon)_+^{p + 1}
= \int_{A_{\varepsilon}} \abs{\nabla (u_{\varepsilon} - q_\varepsilon)}^2.
\end{gather}
\end{lemma}

Such integral identities go back to Berger and Fraenkel \cite{BeFr74}*{lemma~5.A}.

\begin{proof}[Proof of lemma~\ref{lemmaIntegralIdentity}]
The proof goes by taking \((u_{\varepsilon} - q_\varepsilon)_+\) and \(\min (u_{\varepsilon}, q_\varepsilon)\) as test functions in the equation.
\end{proof}

We now study the properties of the vortex core.

\begin{lemma} 
\label{lemmaAeConnecte}
For every \(\varepsilon >0\), the set \(A_{\varepsilon}\) is connected and simply connected and 
\[
 \lim_{\varepsilon \to 0} \frac{\diam (A_{\varepsilon})}{\dist (A_{\varepsilon}, \partial \Omega)} = 0.
\]
\end{lemma}

The proof of the connectedness will require the next techical lemma:

\begin{lemma}
\label{lemmaTruncPos}
Let \(u \in H^1_0 (\Omega, b)\).
If \(u \in C (\Omega)\), if \(U \subset \Omega\) is open,
\[
 \{x \in \Omega : u (x) > 0\} \setminus U
\]
is open and \(\Bar{U} \subset \Omega\) is compact, then 
 \(u_+ \charfun{U} \in H^1_0 (\Omega, b)\).
\end{lemma}

Note that we are not assuming that \(u\) is continuous on \(\partial \Omega\); this makes the proof and the assumptions delicate but will relieve us later of studying the regularity of \(u\) near \(\partial \Omega\).

\begin{proof}[Proof of lemma~\ref{lemmaTruncPos}]
Let \(\delta > 0\) and define 
\[
  K^\delta = \{ x \in U : u (x) \gt \delta \}.
\]
By our assumptions on the function \(u\) and on the sets \(U\), the set \(K^\delta\) is compact.
Hence there exists \(\varphi^\delta \in C^\infty (\Omega; [0, 1])\) such that \(\varphi^\delta = 1\) on \(F_1^\delta\) and \(\varphi^\delta = 0\) on \(\supp u \setminus U\).
One has \( (u - \delta)_+ \charfun{U} = (\varphi^\delta u - \delta)_+ \in H^1_0 (\Omega, b)\).
Since for every \(\delta > 0\), 
\[
 \int_{U} \frac{\abs{\nabla (u - \delta)_+}^2}{b}
\lt \int_{U} \frac{\abs{\nabla u}^2}{b},
\]
we conclude by letting \(\delta \to 0\) that \(u_+ \charfun{U} \in H^1_0 (\Omega, b)\).
\end{proof}

For the connectedness, we rely on an argument that goes back to Berger and Fraenkel \cite{BeFr80}*{theorem 4.3} (see also \citelist{\cite{LiYaYa05}\cite{BeBr80}*{appendix}}). 

\begin{proof}[Proof of lemma~\ref{lemmaAeConnecte}]
Since \(u_{\varepsilon} \gt q_\varepsilon\) on \(A_\varepsilon\), we have by definition of capacity, by lemma~\ref{lemmaChangeWeight} and by lemma~\ref{lemmeGradientEnergie}
\[
\begin{split}
 \inf_\Omega \frac{q_\varepsilon^2}{b} \capa (A_\varepsilon, \Omega) 
  \lt   \int_{\Omega \setminus A_\varepsilon} \frac{q_\varepsilon^2}{b} \Bigabs{\nabla \Bigl(\frac{u_\varepsilon}{q_\varepsilon}\Bigr)}^2
  \lt \int_{\Omega} \frac{\abs{\nabla u_\varepsilon}^2}{b}
  \lt \frac{2 (p + 1)}{p - 1} \mathcal{E}_{\varepsilon} (u_{\varepsilon}).
\end{split}
\]
Let \(A_{\varepsilon}^*\) be a connected component of \(A_\varepsilon\).
Since \(\R^2 \setminus \Omega\) is connected and unbounded, by estimates on the capacity \cite{SmVS10}*{proposition~A.3} (see also \cite{Fr81}),
\[
 \capa (A_{\varepsilon}^*, \Omega) \gt \frac{2 \pi}{\log 16 (1 + \frac{2 \dist (A_\varepsilon, \partial \Omega)}{\diam A_{\varepsilon}^*})}.
\]
In particular \(\Bar{A_{\varepsilon}^*}\) is a compact subset of \(\Omega\) and by proposition~\ref{ThUpperBound},
\[
 \lim_{\varepsilon \to 0} \frac{\diam (A_\varepsilon^*)}{\dist(A_{\varepsilon}^*, \partial \Omega)}  = 0.
\]
It is thus sufficient to prove that
\(A_{\varepsilon}^* = A_{\varepsilon}\).

By lemma~\ref{lemmaTruncPos}, since \(u_{\varepsilon}\) is continuous and \(\Bar{A_{\varepsilon}^*}\) is a compact subset of \(\Omega\),
\begin{equation*}
v_{\varepsilon} = (u_{\varepsilon} - q_\varepsilon)_+ \charfun{A_{\varepsilon}^*} \in H^1_0 (\Omega, b).
\end{equation*}
Also define \(w_\varepsilon = \min (u_{\varepsilon}, q_\varepsilon)\).
By testing the equation against \((u_{\varepsilon} - q_\varepsilon)_+\) and \(v_{\varepsilon}\) we have
\begin{equation}
\label{eqtestAei}
 \int_{A_{\varepsilon}} \abs{\nabla (u_{\varepsilon} - q_\varepsilon)}^2
 = \int_{A_{\varepsilon}} (u_{\varepsilon} - q_\varepsilon)^{p + 1} 
\text{ and }
 \int_{A_{\varepsilon}^*} \abs{\nabla (u_{\varepsilon} - q_\varepsilon)}^2
 = \int_{A_{\varepsilon}^*} (u_{\varepsilon} - q_\varepsilon)^{p + 1}.
\end{equation}
Also note that 
\[
  \int_{\Omega} \abs{\nabla u_{\varepsilon}}^2 = \int_{\Omega} \abs{\nabla w_\varepsilon}^2 + \int_{A_{\varepsilon}} \abs{\nabla (u_{\varepsilon} - q_\varepsilon)}^2,
\]
and for every \(t \in \R\), 
\[
  \int_{\Omega} \abs{\nabla (w_\varepsilon + t v_{\varepsilon})}^2
=   \int_{\Omega} \abs{\nabla w_\varepsilon}^2 + t^2 \int_{A_{\varepsilon}^*} \abs{\nabla  (u_{\varepsilon} - q_\varepsilon)}^2.
\]

We first claim that there exists \(t_* \gt 1\) such that \(w_\varepsilon + t_* u_{\varepsilon} \in \nehari\).
Indeed, one has for every \(t \in \R\),
\[
\begin{split}
  \dualprod{ \mathcal{E}_{\varepsilon}' (w_{\varepsilon} + t v_{\varepsilon})}{ w_{\varepsilon} + t v_{\varepsilon}}
=  &\dualprod{ \mathcal{E}_{\varepsilon}' (w_{\varepsilon} + t v_{\varepsilon})}{ w_{\varepsilon} + t v_{\varepsilon}}
- \dualprod{ \mathcal{E}_{\varepsilon}' (u_{\varepsilon})}{ u_{\varepsilon} }\\
= & t^2 \int_{A_{\varepsilon}^*} \abs{\nabla (u_{\varepsilon} - q_\varepsilon)}^2 - \int_{A_{\varepsilon}} \abs{\nabla (u_{\varepsilon} - q_\varepsilon)}^2\\
&- \int_{A_{\varepsilon}^*} t^{p} (u_{\varepsilon} - q_\varepsilon)^p (q_\varepsilon + t(u_{\varepsilon} - q_\varepsilon))
+ \int_{A_{\varepsilon}}  (u_{\varepsilon} - q_\varepsilon)^p u_{\varepsilon}.
\end{split}
\]
By \eqref{eqtestAei}, we have
\begin{multline*}
  \dualprod{ \mathcal{E}_{\varepsilon}' (w_{\varepsilon} + t u_{\varepsilon}^1)}{ w_{\varepsilon} + t u_{\varepsilon}^1}
= \int_{A_{\varepsilon} \setminus A_{\varepsilon}^*}  (u_{\varepsilon} - q_\varepsilon)^p q_\varepsilon
 - (t^{p + 1} - t^2) \int_{A_{\varepsilon}^*}  (u_{\varepsilon} - q_\varepsilon)^{p + 1}\\
- (t^{p} - 1) \int_{A_{\varepsilon}^*}  (u_{\varepsilon} - q_\varepsilon)^p q_\varepsilon.
\end{multline*}
By the intermediate value theorem, there exists \(t_* \gt 1\) such that \(w_{\varepsilon} + t_* u_{\varepsilon} \in \nehari\).

Now we compute the energy and we obtain
\[
\begin{split}
 \mathcal{E}_{\varepsilon} (w_\varepsilon + t_* u_{\varepsilon})
&= \frac{1}{2} \int_{\Omega} \abs{\nabla w_\varepsilon}^2 + \Bigl(\frac{t_*^2}{2} - \frac{t_*^{p + 1}}{p + 1}\Bigr) \int_{A_{\varepsilon}^*} (u_{\varepsilon} - q_\varepsilon)^{p + 1}\\
&\lt \mathcal{E}_{\varepsilon} (u_{\varepsilon}) - \Bigl(\frac{1}{2} - \frac{1}{p + 1}\Bigr) \int_{A_{\varepsilon} \setminus A_{\varepsilon}^*} (u_{\varepsilon} - q_\varepsilon)^{p + 1}.
\end{split}
\]
Since \(u_{\varepsilon}\) is a minimal energy solution and \(u_{\varepsilon} > q_\varepsilon\) in \(A_{\varepsilon}\), we conclude that \(A_{\varepsilon} = A_{\varepsilon}^*\) and the set \(A_{\varepsilon}\) is thus connected. 

We now show that \(A_{\varepsilon}\) is simply connected. Let \(E\) be the connected component of \(\Omega \setminus A_{\varepsilon}\) such that \(\partial \Omega \subset \Bar{E}\). The set \(\Omega \setminus E\) is open and one has  \(- \dive (\frac{u_{\varepsilon}-q_\varepsilon}{b}) \gt 0\) in \(\Omega \setminus E\) and \(u_{\varepsilon}-q_\varepsilon \gt 0\) in \(\Omega \setminus E\), so that by the strong maximum principle, \(u_{\varepsilon} - q_\varepsilon > 0\) in \(\Omega \setminus E\). Hence \(A_{\varepsilon} = \Omega \setminus E\) and \(A_{\varepsilon}\) is simply connected.
\end{proof}

The next lemma shows that the kinetic energy remains bounded inside the vortex core.

\begin{proposition}\label{ThBornesGrandeursPhysiques}
There exists a constant \(C>0\) independent of \(\varepsilon\) such that if \(a_{\varepsilon} \in A_{\varepsilon}\),
\begin{equation*}
\frac{1}{\varepsilon^2} \int_{A_{\varepsilon}} b (u_\varepsilon - q_\varepsilon)_+^{p + 1} = \int_{A_{\varepsilon}} \frac{\abs{\nabla (u_{\varepsilon} - q_\varepsilon)}^2}{b}  \lt C \frac{b (a_{\varepsilon})}{q (a_{\varepsilon})^2}.
\end{equation*}
\end{proposition}
\begin{proof}
Let \(a_{\varepsilon} \in A_{\varepsilon}\). By lemma~\ref{lemmaIntegralIdentity} \eqref{eqIntIDInterior}, lemma~\ref{lemmaAeConnecte} and \((\mathcal{A}_1)\), one has 
\begin{align*}
\int_{A_{\varepsilon}} \frac{\abs{\nabla (u_{\varepsilon} - q_\varepsilon)}^2}{b} &  = \frac{1}{\varepsilon^2} \int_{A_{\varepsilon}}  b (u_{\varepsilon} - q_\varepsilon)_+^{p+1} \\
& \lt C b(a_{\varepsilon}) \frac{1}{\varepsilon^2} \int_{A_{\varepsilon}}  (u_{\varepsilon} - q_\varepsilon)_+^{p+1} \\
& \lt C' b(a_{\varepsilon}) \frac{1}{\varepsilon^2} \int_{A_{\varepsilon}} (u_{\varepsilon} - q_\varepsilon)_+^{p} \ \Bigl( \int_{A_{\varepsilon}} \abs{\nabla (u_{\varepsilon} - q_\varepsilon)}^2 \Bigr)^{1/2} \\
& \lt C''  b(a_{\varepsilon})^{1/2}  \frac{1}{\varepsilon^2} \int_{A_{\varepsilon}}  b (u_{\varepsilon} - q_\varepsilon)_+^{p} \ \Bigl( \int_{A_{\varepsilon}}  \frac{\abs{\nabla (u_{\varepsilon} - q_\varepsilon)}^2}{b} \Bigr)^{\frac{1}{2}},
\end{align*}
using the the classical Gagliardo-Nirenberg inequality. One obtains thus
\begin{equation*}
 \int_{A_{\varepsilon}} \frac{\abs{\nabla (u_{\varepsilon} - q_\varepsilon)}^2}{b} \lt (C'')^2 b(a_{\varepsilon}) \left( \frac{1}{\varepsilon^2}  \int_{A_{\varepsilon}}  b (u_{\varepsilon} - q_\varepsilon)_+^{p} \right)^2.
\end{equation*}
Now by lemma~\ref{lemmaAeConnecte} and  by lemma~\ref{lemmaIntegralIdentity} \eqref{eqIntIDExterior},
\[
 q (a_{\varepsilon}) \frac{1}{\varepsilon^2}  \int_{A_{\varepsilon}}  b (u_{\varepsilon} - q_\varepsilon)_+^{p}
 \lt C''' \frac{1}{\varepsilon^2}  \int_{A_{\varepsilon}}  b (u_{\varepsilon} - q_\varepsilon)_+^{p} q
 \lt C''' \frac{1}{\log \tfrac{1}{\varepsilon}} \int_{\Omega} \frac{\abs{\nabla u_{\varepsilon}}^2}{b},
\]
and we conclude by lemma~\ref{lemmeGradientEnergie} and proposition~\ref{ThUpperBound}.
\end{proof}

Finally, we have a lower bound on the diameter of the vortex core:

\begin{lemma}\label{ThBorneInfDiametreAe}
There exists a constant \(C>0\) such that if \(a_{\varepsilon} \in A_{\varepsilon}\),
\begin{equation*}
\diam(A_{\varepsilon}) \gt \frac{C \varepsilon q (a_{\varepsilon})^{\frac{p - 1}{2}}}{b (a_{\varepsilon})^{\frac{p + 1}{2}}}.
\end{equation*}
\end{lemma}

\begin{proof}
One has, by lemma~\ref{lemmaAeConnecte} and  and \((\mathcal{A}_1)\),
\begin{equation*}
\frac{1}{\varepsilon^2} \int_{A_{\varepsilon}}  b (u_{\varepsilon} - q_\varepsilon)_+^{p+1}
\lt C b(a_{\varepsilon}) \frac{1}{\varepsilon^2} \int_{A_{\varepsilon}}  (u_{\varepsilon} - q_\varepsilon)_+^{p+1} . 
\end{equation*}
By the Hölder and Sobolev inequalities
\begin{equation*}
 \int_{A_{\varepsilon}}  (u_{\varepsilon} - q_\varepsilon)_+^{p+1} \lt C' \abs{A_{\varepsilon}} \Bigl( \int_{A_{\varepsilon}} \abs{\nabla (u_{\varepsilon} - q_\varepsilon)}^2 \Bigr)^{(p+1)/2}. 
\end{equation*}
Hence we obtain, by lemma~\ref{lemmeGradientEnergie} and lemma~\ref{lemmaAeConnecte} together with \((\mathcal{A}_1)\) again,
\begin{equation*}
 \int_{A_{\varepsilon}} \frac{\abs{\nabla (u_{\varepsilon} - q_\varepsilon)}^2}{b} \lt C'' b(a_{\varepsilon})^{(p + 3)/2} \frac{\abs{A_{\varepsilon}}}{\varepsilon^2}  \Bigl( \int_{A_{\varepsilon}} \frac{\abs{\nabla (u_{\varepsilon} - q_\varepsilon)}^2}{b} \Bigr)^{(p+1)/2}. 
\end{equation*}
Therefore,
\[
  \liminf_{\varepsilon \to 0} b(a_{\varepsilon})^\frac{p + 3}{2} \frac{\abs{A_{\varepsilon}}}{\varepsilon^2} \Bigl(\int_{A_{\varepsilon}} \frac{\abs{\nabla (u_{\varepsilon} - q_\varepsilon)}^2}{b}\Bigr)^\frac{p - 1}{2} > 0.
\]
By proposition~\ref{ThBornesGrandeursPhysiques}, this implies that 
\[
   \liminf_{\varepsilon \to 0} \frac{\abs{A_{\varepsilon}}}{\varepsilon^2} \frac{b (a_{\varepsilon})^{p + 1}}{q (a_{\varepsilon})^{p - 1}} > 0. 
\]
and the result follows from the isodiametric inequality \(\abs{A_{\varepsilon}} \lt \pi (\diam A_{\varepsilon})^2/4\).
\end{proof}

The main result of this section is:

\begin{proposition}
\label{propositionAsymptotics}
One has, if \(a_{\varepsilon} \in A_{\varepsilon}\),
\begin{gather}
\label{EqXeConvergeVersArgMin}
\tag{a} \lim_{\varepsilon \to 0} \frac{\mathcal{E}_{\varepsilon}(u_{\varepsilon})}{\pi \log \tfrac{1}{\varepsilon}} 
= \lim_{\varepsilon \to 0} \frac{1}{2 \pi \log \tfrac{1}{\varepsilon}} \int_{\Omega} \frac{\abs{\nabla u_{\varepsilon} }^2 }{b} 
= \lim_{\varepsilon \to 0} \frac{q(a_{\varepsilon})^2}{b(a_{\varepsilon})} 
= \inf_{\Omega}  \frac{q^2}{b},\\
\tag{b} \label{EqKappaBorneInf}
  \lim_{\varepsilon \to 0} \kappa_{\varepsilon} \frac{b (a_{\varepsilon})}{q (a_{\varepsilon})} = 2\pi,\\
\label{EqBornesInfSupDiametre}
\tag{c} 
\lim_{\varepsilon \to 0} 
    \frac{\log \dfrac{\dist (A_{\varepsilon}, \partial \Omega)}{\diam (A_{\varepsilon})}}
         {\log \dfrac{1}{\varepsilon}}
= \lim_{\varepsilon \to 0} 
   \frac{\log \dfrac{b(a_{\varepsilon})^{(p + 1)/2}}{\diam (A_{\varepsilon}) q (a_{\varepsilon})^{(p - 1)/2}}}
        {\log \dfrac{1}{\varepsilon}} =  1. 
\end{gather}
\end{proposition}

\begin{proof}
By definition of \(\mathcal{E}_{\varepsilon}\) and by proposition~\ref{lemmaIntegralIdentity}, we have
\begin{equation*}
\begin{split}
\label{equeqeqe} \frac{1}{\varepsilon^2} \int_{\Omega} b (u_{\varepsilon} - q_\varepsilon)_+^p q_\varepsilon 
& = \int_{\Omega} \frac{\abs{ \nabla u_{\varepsilon}}^2}{b} - \int_{A_{\varepsilon}} \frac{\abs{ \nabla (u_{\varepsilon} - q_\varepsilon)}^2}{b}\\
&= 2\mathcal{E}_{\varepsilon}(u_{\varepsilon}) - \frac{p - 1}{p+1} \frac{1}{\varepsilon^2} \int_{A_{\varepsilon}}  b (u_{\varepsilon} - q_\varepsilon)_+^{p+1}.
\end{split}
\end{equation*}
Hence, by proposition~\ref{ThBornesGrandeursPhysiques},
\begin{equation}
\label{eqbuqEe}
\int_{\Omega} b (u_{\varepsilon} - q_\varepsilon)_+^p q \lt \frac{2 \mathcal{E}_{\varepsilon}(u_{\varepsilon})}{\log \tfrac{1}{\varepsilon}} + O \Bigl(\frac{1}{\log \tfrac{1}{\varepsilon}}\Bigr),
\end{equation}
as \(\varepsilon \to 0\).
Define for \(\sigma, \tau \in (0, 1)\) with \(\tau < \sigma\),
\[
 w_{\varepsilon}^{\sigma, \tau} = \min\Bigl(\frac{(u_{\varepsilon} - q_{\sigma})_+}{q_{\tau} - q_{\sigma} }, 1\Bigr).
\]
By testing the equation against \(w_{\varepsilon}^{\sigma, \tau} q\), in view of lemma~\ref{lemmaChangeWeight}
\[
 \log \frac{\sigma}{\tau} \int_{\Omega} \frac{q^2}{b} \abs{\nabla w_{\varepsilon}^{\sigma, \tau}}^2
 = \frac{1}{\varepsilon^2} \int_{\Omega} (u_{\varepsilon} - q_\varepsilon)_+^p q.
\]
In particular, setting
\[
  A_{\varepsilon}^\tau = \bigl\{x \in \Omega : u_{\varepsilon} (x) > q_{\tau} (x)\bigr\},
\]
we have
\[
 \capa (A_{\varepsilon}^{\tau}, \Omega) \lt \frac{\displaystyle \int_{\Omega} \frac{q^2}{b} \abs{\nabla w_{\varepsilon}^{1, \tau}}^2}{\displaystyle \inf_{\Omega} \frac{q^2}{b}}
\]
and thus by capacity estimates \cite{SmVS10}*{proposition~A.3} (see also \cite{Fr81}), since \(\R^2 \setminus \Omega\) is unbounded and connected,
\[
  \frac{2\pi}{\displaystyle \log 16 \Bigl(1 + \frac{2 \dist (A_{\varepsilon}^{\tau}, \partial \Omega)}{\diam A_{\varepsilon}^{\tau}}\Bigr)}
\lt \frac{\displaystyle \frac{1}{\varepsilon^2} \int_{\Omega} b (u_{\varepsilon} - q_\varepsilon)_+^p q}{\log \frac{1}{\tau} \displaystyle \inf_{\Omega} \frac{q^2}{b}}. 
\]

By \eqref{eqbuqEe} and by proposition~\ref{ThUpperBound}, we have
\begin{equation}
\label{eqliminfAetau}
 \liminf_{\varepsilon \to 0} \log 16 \Bigl(1 + \frac{2 \dist (A_{\varepsilon}^{\tau}, \partial \Omega)}{\diam A_{\varepsilon}^{\tau}}\Bigr) \gt \log \frac{1}{\tau},
\end{equation}
and thus, by \((\mathcal{A}_1)\), for every \(\delta > 0\), there exists \(\rho > 0\) and \(\varepsilon_0 > 0\) such that for every \(x, y \in A_{\varepsilon}^{\rho}\) with \(\varepsilon \in (0, \varepsilon_0)\),
\[
 \frac{q (x)^2}{b (x)}
\lt \frac{q (y)^2}{b (y)} (1 + \delta).
\]
We have thus
\begin{equation}
\label{eqwete}
 \frac{q (a_{\varepsilon})^2}{b (a_{\varepsilon})(1 + \delta)}   \int_{\Omega} \abs{\nabla w_{\varepsilon}^{\tau, \varepsilon}}^2
\lt \int_{\Omega} \frac{q^2}{b} \abs{\nabla w_{\varepsilon}^{\tau, \varepsilon}}^2
 \lt \frac{1}{\log \frac{\tau}{\varepsilon}} \frac{1}{\varepsilon^2} \int_{\Omega} (u_{\varepsilon} - q_\varepsilon)_+^p q .
\end{equation}
By capacity estimates, we have thus that for every \(\varepsilon \in (0, \varepsilon_0)\),
\[
  \int_{\Omega} \abs{\nabla w_{\varepsilon}^{\tau, \varepsilon}}^2 \gt \capa (A_\varepsilon, \Omega) \gt  \frac{2 \pi}{\log 16 (1 + \frac{2 \dist (A_\varepsilon, \partial \Omega)}{\diam (A_\varepsilon)})}
\]
and hence
\[
 \frac{q (a_{\varepsilon})^2}{b (a_{\varepsilon})} \frac{2 \pi}{\log 16 (1 + \frac{2 \dist (A_\varepsilon, \partial \Omega)}{\diam (A_\varepsilon)})}
\lt \frac{1+ \delta}{\varepsilon^2} \int_{\Omega} (u_{\varepsilon} - q_\varepsilon)_+^p q.
\]
In view of lemma~\ref{ThBorneInfDiametreAe}, we have
\[
 \limsup_{\varepsilon \to 0} \frac{q (a_{\varepsilon})^2}{b (a_{\varepsilon})} \frac{\log \frac{\tau}{\varepsilon}}{\log 16 \Bigl(1 + \dfrac{C \dist (a_{\varepsilon}, \partial \Omega) b (a_{\varepsilon})^{(p + 1)/2}}{\varepsilon q (a_{\varepsilon})^{(p - 1)/2}}\Bigr)}
\lt (1 + \delta)  \inf_{\Omega} \frac{q^2}{b}.
\]

Now, note that
\begin{multline*}
 \log 16 \Bigl(1 + \dfrac{C \dist (a_{\varepsilon}, \partial \Omega) b (a_{\varepsilon})^{(p + 1)/2}}{\varepsilon q (a_{\varepsilon})^{(p - 1)/2}}\Bigr)\\
\lt \log 16 \Bigl(1 + \frac{C}{\varepsilon}\Bigr) + \log \Bigl(1 + \dfrac{\dist (a_{\varepsilon}, \partial \Omega) b (a_{\varepsilon})^{(p + 1)/2}}{q (a_{\varepsilon})^{(p - 1)/2}}\Bigr).
\end{multline*}
By assumption \((\mathcal{A}_2)\), we have thus
\begin{multline*}
 \limsup_{\varepsilon \to 0} \frac{q (a_{\varepsilon})^2}{b (a_{\varepsilon})} \frac{\log \frac{\tau}{\varepsilon}}{\log 16 \Bigl(1 + \dfrac{2 \dist (a_{\varepsilon}, \partial \Omega) b (a_{\varepsilon})^{(p + 1)/2}}{\varepsilon q (a_{\varepsilon})^{(p - 1)/2}}\Bigr)}\\
 \gt \limsup_{\varepsilon \to 0} \frac{\dfrac{\log \frac{\tau}{\varepsilon}}{\log 16 (1 + \frac{C}{\varepsilon})}}
{\dfrac{b (a_{\varepsilon})}{q (a_{\varepsilon})^2} + \dfrac{C'}{\log 16 ( 1 + \frac{C}{\varepsilon})}}
\gt  \limsup_{\varepsilon \to 0} \frac{q (a_{\varepsilon})^2}{b (a_{\varepsilon})}.
\end{multline*}
Hence, we conclude that 
\[
 \limsup_{\varepsilon \to 0} \frac{q (a_{\varepsilon})^2}{b (a_{\varepsilon})} \lt (1 + \delta) \inf_{\Omega} \frac{q^2}{b}.
\]
Since \(\delta > 0\) is arbitrary, we have \eqref{EqXeConvergeVersArgMin}.

To obtain \eqref{EqBornesInfSupDiametre}, we note that 
\[
 \limsup_{\varepsilon \to 0} 
   \frac{\log \dfrac{\dist (A_{\varepsilon}, \partial \Omega)}{\diam (A_{\varepsilon})}}
        {\log \dfrac{1}{\varepsilon}} 
\lt 1
\]
and that by lemma~\ref{ThBorneInfDiametreAe},
\[
 \liminf_{\varepsilon \to 0} 
    \frac{\log \dfrac{ b(a_{\varepsilon})^\frac{p + 1}{2}}{\diam (A_{\varepsilon}) q(a_{\varepsilon})^\frac{p - 1}{2}}}
         {\log \dfrac{1}{\varepsilon}} 
\lt 1.
\]
We conclude since by \((\mathcal{A}_2)\) 
\[
\begin{split}
 \limsup_{\varepsilon \to 0}
   \frac{\log \dfrac{\dist (A_{\varepsilon}, \partial \Omega)}{\diam (A_{\varepsilon})}}
        {\log \dfrac{1}{\varepsilon}} 
-    \frac{\log \dfrac{ b(a_{\varepsilon})^\frac{p + 1}{2}}{\diam (A_{\varepsilon}) q(a_{\varepsilon})^\frac{p - 1}{2}}}
         {\log \dfrac{1}{\varepsilon}} 
&= \limsup_{\varepsilon \to 0}
   \frac{\log \dfrac{\dist (a_{\varepsilon}, \partial \Omega)b(a_{\varepsilon})^\frac{p + 1}{2}}
                    {\diam (A_{\varepsilon}) q(a_{\varepsilon})^\frac{p - 1}{2}}}
        {\log \dfrac{1}{\varepsilon}}\\
& \lt \limsup_{\varepsilon \to 0} 
       \frac{C' \dfrac{q (a_{\varepsilon})^2}{b (a_{\varepsilon})}}
            {\log \tfrac{1}{\varepsilon}}
= 0.
\end{split}  
\]

To obtain \eqref{EqKappaBorneInf}, note that by \eqref{equeqeqe} and by lemma~\ref{ThBornesGrandeursPhysiques}, we have
\[
  \lim_{\varepsilon \to 0}
\frac{1}{\varepsilon^2} \int_{A_{\varepsilon}} b (u_{\varepsilon}-q_\varepsilon)_+^p q = 2\pi \inf_{\Omega} \frac{q^2}{b}
\]
and by lemma~\ref{lemmaAeConnecte}, we have
\[
 \lim_{\varepsilon \to 0} \int_{A_{\varepsilon}} b (u_{\varepsilon}-q_\varepsilon)_+^p q - q (a_{\varepsilon}) \int_{A_{\varepsilon}} (u_{\varepsilon} - q_\varepsilon)_+^p = 0.\qedhere
\]
\end{proof}

If \(\frac{q^2}{b}\) is Dini-continuous and if the solutions concentrate around an interior point, we have the following improvement.

\begin{proposition}
\label{propositionImprovedAsymptotics}
If \(\lim_{\varepsilon \to 0} a_{\varepsilon} = \Hat{a} \in \Omega\) and \(\frac{q^2}{b}\) is Dini-continuous in a neighbourhood of \(\Hat{a}\),
then
\begin{gather*}
\frac{\mathcal{E}_{\varepsilon}(u_{\varepsilon})}{\pi}
= \frac{1}{2 \pi} \int_{\Omega} \frac{\abs{\nabla u_{\varepsilon} }^2 }{b} + O (1)
= \frac{q(a_{\varepsilon})^2}{b(a_{\varepsilon})}\log \tfrac{1}{\varepsilon} + O (1)
= \inf_{\Omega} \frac{q^2}{b}\log \tfrac{1}{\varepsilon} + O (1),\\
   0 < \liminf_{\varepsilon \to 0} \frac{\diam A_{\varepsilon}}{\varepsilon} \lt \limsup_{\varepsilon \to 0} \frac{\diam A_{\varepsilon}}{\varepsilon} < \infty.
\end{gather*}
\end{proposition}

\begin{proof}[Sketch of the proof of proposition~\ref{propositionImprovedAsymptotics}]
Beginning as in the proof of proposition~\ref{propositionAsymptotics}, we have, as proposition~\ref{propImprovedUpper} is applicable in place of \eqref{eqliminfAetau}, that there exists \(c > 0\) such that 
\[
 \Bigl(1 + \frac{2 \dist (A_{\varepsilon}^{\tau}, \partial \Omega)}{\diam A_{\varepsilon}^{\tau}}\Bigr) \gt \frac{c}{\tau}.
\]
Since \(\lim_{\varepsilon \to 0} a_{\varepsilon} = \Hat{a} \in \Omega\), there exists \(\rho > 0\) such that for every \(\sigma \gt \rho\) and \(x \in A_{\varepsilon}^\sigma\), \(\abs{x - a_{\varepsilon}} \lt C \sigma\).
This implies that 
\[
 \Bigl(\log \frac{\sigma}{\tau}\Bigr)^2 \int_{\Omega} (u_{\varepsilon} - q_\varepsilon)^p_+ q \lt  \Bigl(\log \frac{\sigma}{\tau}\Bigr)
\frac{q (\Hat{a})^2}{b (\Hat{a})} \bigl(1 + \omega (C \tau)\bigr) \frac{1}{\varepsilon^2} \int_{\Omega} (u_{\varepsilon} - q_\varepsilon)^p_+ q.
\]
Taking now \(\varepsilon = \tau_1 < \sigma_1 = \tau_2 < \sigma_2 < \dotsc < \sigma_k = \rho\) and  summing the previous inequality over \(j \in \{1, \dotsc, k\}\), we obtain
\[
  \frac{q (\Hat{a})^2}{b (\Hat{a})} \Bigl(\log \frac{\rho}{\varepsilon}\Bigr)^2 \int_{\Omega} \abs{\nabla w_{\varepsilon}^{\rho, \varepsilon}}^2
\lt \Bigl( \log \frac{\rho}{\varepsilon}+ \sum_{j = 1}^k \omega (C \sigma_i)  \log \frac{\sigma_i}{\tau_i}\Bigr) \frac{1}{\varepsilon^2} \int_{\Omega} (u_{\varepsilon} - q_\varepsilon)^p_+ q.   
\]
By taking the limit of Riemann sums, we conclude that 
\[
 \frac{q (\Hat{a})^2}{b (\Hat{a})}  \int_{\Omega} \abs{\nabla w_{\varepsilon}^{\rho, \varepsilon}}^2 \lt 
\Bigl(\frac{1}{\log \frac{\rho}{\varepsilon}} + \frac{1}{(\log \frac{\rho}{\varepsilon})^2} \int_\varepsilon^\rho \frac{\omega (C \tau)}{\tau} \du \tau \Bigr) \frac{1}{\varepsilon^2} \int_{\Omega} (u_{\varepsilon} - q_\varepsilon)^p_+ q,
\]
which improves \eqref{eqwete} and allows to continue the proof.
\end{proof}

\section{Construction and asymptotics of vortices}

In this section we go back to the axisymmetric Euler equation and the shallow water equation and prove our main results.

\subsection{Vortex rings for the Euler equation}
For the Euler equation, the solutions of the previous sections gives us a suitable Stokes stream functions.

\subsubsection{Vortex ring in the whole space}
The first case is the construction of a vortex ring in the whole space.

\begin{proof}[Proof of theorem~\ref{theoremVortexRingSpace}]
Define for every \(r \in (0, \infty)\) and \(z \in \R\), \(b (r, z) = r\) and 
\[
 q (r, z) = W \frac{r^2}{2} + \frac{3}{8 W}\Bigl(\frac{\kappa}{2 \pi}\Bigr)^2.
\]
One computes directly that \(\frac{q^2}{b}\) achieves its minimum at \( (\frac{\kappa}{4 \pi W}, 0)\) and that 
\[
 2 \pi \frac{q (r_*, 0)}{b (r_*, 0)} = \kappa.
\]
By proposition~\ref{propositionExistenceInvariant}, the problem has a solution for every \(\varepsilon \in (0, 1)\). 
Define
\[
 \mathbf{v}_\varepsilon (r, z) = \curl \Bigl((u_{\varepsilon} + q_\varepsilon) \frac{\mathbf{e}_\theta}{r}\Bigr)
\]
and 
\[
 p_\varepsilon (r, z) = \frac{(u_{\varepsilon} - q_\varepsilon)^{p + 1}}{p + 1} - \frac{\abs{\mathbf{v}_\varepsilon}^2}{2}.
\]

One computes that 
\[
 \lim_{\abs{x} \to \infty} \mathbf{v} =  \frac{\kappa}{4 \pi r^*} \log \tfrac{1}{\varepsilon}
\]
and that 
\[
 \curl \mathbf{v}_\varepsilon (r, z) = (u_\varepsilon (r, z) - q_\varepsilon (r, z))_+^p \mathbf{e}_\theta.
\]
The conclusion follows by the asymptotics of propositions~\ref{propositionAsymptotics} and~\ref{propositionImprovedAsymptotics}.
\end{proof}

\subsubsection{Vortex ring in a cylinder}

The proof of theorem~\ref{theoremVortexRingCylinder} is very similar.

\begin{proof}[Proof of theorem~\ref{theoremVortexRingCylinder}]
If \(\kappa < 4 \pi W\), one defines \(b\) and \(q\) as in the proof of theorem~\ref{theoremVortexRingSpace}. 
Otherwise one sets
\[
  q (r, z) = \frac{W}{2} r^2 + \Bigl(\frac{\kappa}{2 \pi} - \frac{W}{2}\Bigr). 
\]
One checks that \(\frac{q^2}{b}\) achieves its minimum at \((1, 0)\).
Since \(\frac{\kappa}{2 \pi} - \frac{W}{2} \gt 0\), we can then use proposition~\ref{propositionAsymptotics} in the asymptotics.
\end{proof}

\subsubsection{Vortex ring outside a ball}
For the construction of a vortex ring outside a ball, we use the strict inequality of proposition~\ref{propositionStrict}.

\begin{proof}[Proof of theorem~\ref{theoremVortexRingOutsideBall}]
If \(\kappa > 6 \pi W\), let \(r_*\) be the unique number such that 
\[
 2 r_* + \frac{1}{r_*^2} = \frac{\kappa}{2 \pi W}.
\]
Define now
\[
 q (r, z) = \frac{W}{2}\Bigl(r^2 - \frac{r}{r^2 + z^2}\Bigr) + \frac{3 W}{2} \Bigl(r_*^2 + \frac{1}{r_*}\Bigr).
\]
Observe that if \(z \ne 0\),
\[
 q (r, z) > q (r, 0).
\]
If the function \( r \in [1, \infty) \mapsto q (r, z)\) achieves its maximum at \(\Tilde{r} \in (1, \infty)\), by Fermat's theorem, 
\[
  \frac{3}{2}\frac{\frac{W}{2}\bigl(\Tilde{r}^2 - \frac{1}{\Tilde{r}}\bigr) + \frac{3 W}{2} \bigl(r_*^2 + \frac{1}{r_*}\bigr)}{r^2}
\Bigl( \Tilde{r}^2 + \frac{1}{\Tilde{r}} - r_*^2 - \frac{1}{r_*}\Bigr) = 0,
\]
from which we deduce that \(\Tilde{r} = r_*\).
Define
\[
  q^\infty (x) = \frac{W}{2} r^2 + \frac{3 W}{2} \Bigl(r_*^2 + \frac{1}{r_*}\Bigr),
\]
and observe that 
\[
 \lim_{\abs{x} \to \infty} \frac{q (x)}{q^\infty (x)} = 1,
\]
and that 
\[
  \inf_{\Omega} \frac{(q^\infty)^2}{b} = \frac{q^\infty (r^\infty_*, 0)^2}{b (r^\infty_*, 0)}
> \frac{q (r^\infty_*, 0)^2}{b (r^\infty_*, 0)}
\]
with
\[
 (r^\infty_*)^2 = r_*^2 + \frac{1}{r_*}.
\]
In particular, since \(r_* \gt 1\), \((r^\infty_*, 0) \in \R^2_+ \setminus B_1\).
By proposition~\ref{ThUpperBound}, we have
\[
 \limsup_{\varepsilon \to \infty} \frac{c_\varepsilon}{\log \tfrac{1}{\varepsilon}} \lt 
\inf_{\R^2 \setminus B_1} \frac{q^2}{b} <  \pi \frac{q^\infty (r^\infty_*, 0)^2}{b (r^\infty_*, 0)}
\]
and by proposition~\ref{propositionAsymptotics}, 
\[
 \lim_{\varepsilon \to \infty} \frac{c^\infty_\varepsilon}{\log \tfrac{1}{\varepsilon}} = \pi \frac{q^\infty (r^\infty_*, 0)^2}{b (r^\infty_*, 0)}.
\]
By proposition~\ref{propositionStrict}, the problem~\eqref{problemP} has a solution. 
One constructs the flow and studies its asymptotics by proposition~\ref{propositionAsymptotics} as in the proof of theorem~\ref{theoremVortexRingSpace}.

If \(\kappa \lt 6 \pi W \), define 
\[
 q (r, z) = \frac{W}{2}\Bigl(r^2 - \frac{r}{r^2 + z^2}\Bigr) + \frac{\kappa}{2 \pi},
\]
and 
\[
 q^\infty (r, z) = \frac{W}{2} r^2 + \frac{\kappa}{2 \pi}
\]
and observe that \(\frac{q^2}{b}\) achieves its maximum at \((1, 0)\) and that
\[
 \inf_{\R^2_+ \setminus B_1} \frac{q^2}{b} = \Bigl(\frac{\kappa}{2 \pi}\Bigr)^2
\]
and 
\[
 \inf_{(r, z) \in \R^2_+ \setminus B_1} \frac{q^\infty (r, z)}{r}
= \frac{16}{9 \sqrt{2}} \sqrt{\frac{6 \pi W}{\kappa}} \Bigl(\frac{\kappa}{2 \pi}\Bigr)^2
> \Bigl(\frac{\kappa}{2 \pi}\Bigr)^2
\]
since \(\kappa \lt 6 \pi W\).
The rest of the proof is similar to the case \(\kappa > 6 \pi W \). 
\end{proof}

\subsubsection{Vortex ring outside a compact set}

In order to construct solutions outside an arbitrary compact set, we first construct and study the irrotational flow.

\begin{lemma}
\label{lemmaLinear}
Let \(\alpha > - 1\), \(k \gt 0\) and \(K \subset \R^2\).
Define \(b : \R^2_+ \to \R\) and \(q^\infty : \R^2_+ \to \R\) be defined for \(x = (x_1, x_2) \in \R^2_+\) by 
\begin{align*}
b (x) &= x_1^\alpha &
& \text{ and }&
  q^\infty (x) &= \tfrac{W}{\alpha + 1} x_1^{\alpha + 1} + k.
\end{align*}
If \(K\) is compact and satisfies an interior cone condition at every point of \(\partial K \cap \R^2_+\), then there exists a unique solution \(q \in H^1_\mathrm{loc} (\R^2_+ \setminus K) \cap C (\overline{\R^2_+ \setminus K})\) such that 
\begin{equation*}
\left\{
\begin{aligned}
   - \dive \frac{\nabla q}{b} & = 0 & & \text{in \(\R^2_+ \setminus K\)},\\
   q   & =  k & & \text{on \(\partial (\R^2_+ \setminus K)\)},\\
   \lim_{\abs{x} \to \infty} \frac{q (x)}{q_\infty (x)} & = 1.
\end{aligned}
\right.
\end{equation*}
Moreover \(q \in C^\infty (\R^2_+)\), 
\[
  \lim_{\abs{x} \to \infty} \frac{\nabla q (x)}{ x_1^{\alpha}} = (W, 0),
\]
and, if \(K \cap \R^2_+ \ne \emptyset\), for every \(x \in \R^2_+ \setminus K\),
\[
  q (x) < q^\infty (x).
\]
\end{lemma}

\begin{proof}
Since \(K\) is compact there exists \(R > 0\) such that \(K \subset B (0, R)\).
Choose \(\varphi \in C^\infty (\R^2)\) so that \(\varphi = 1\) on \(B (0, R)\) and \(\varphi = 0\) in \(\R^2 \setminus B (0, 2 R)\) and define \(g : \R^2_+ \to \R\) for \(x \in \R^2_+\) by 
\[
 g (x) =  \frac{W}{\alpha + 1} \varphi (x)  x_1^{\alpha + 1}.
\]
Observe that since \(\alpha > - 1\),
\[
 \int_{\R^2_+} \frac{\abs{\nabla g}^2}{b}
\lt 2 \Bigl(\frac{W}{\alpha + 1}\Bigr)^2 \int_{\R^2_+} \abs{\nabla \varphi (x)}^2 x_1^{\alpha + 2} + (\alpha + 1)^2  \abs{\varphi (x)}^2 x_1^{\alpha} \du x < \infty.
\]
Construct the function \(v \in H^1_0 (\R^2_+ \setminus K)\) by minimizing the Dirichlet energy
\begin{equation*}
  \frac{1}{2} \int_{\Omega} \frac{\abs{\nabla v}^2}{b} - \int_{\Omega} \frac{\nabla g \cdot \nabla v}{b}
\end{equation*}
over \(H^1_0 (\Omega, b)\) and set 
\[
q = q_\infty - g + v. 
\]

One has clearly \(v \in H^1_\mathrm{loc} (\R^2_+ \setminus K)\) and 
\[
  \dive \frac{\nabla q}{b} = 0
\]
weakly in \(\R^2_+ \setminus K\). By the classical interior regularity theory, \(v \in C^\infty (\R^2_+ \setminus K)\).
Since \(K \cap \R^2_+\) satisfies an interior cone condition at every point of \(\partial K \cap \R^2_+\), \(v\) is continous on 
\(\R^2_+ \setminus \inter K\) \cite{GiTr01}*{Corollary 8.28}.

Now we claim that \(v \lt g\). Indeed, by taking \((v - g)_+ \in H^1_0 (\Omega, b)\) as a test function in the equation, we have
\[
  \int_{\R^2_+ \setminus K} \abs{\nabla (v - g)_+}^2 = \int_{\R^2_+ \setminus K} (\nabla v - \nabla g) \cdot \nabla (v - g)_+ = 0,
\]
so that \(v \lt g\). In particular, we have \(q \lt q_\infty\). Similarly, one has that \(v \gt k + g - q_\infty\), so that we have proved that 
\begin{equation}
\label{eqLinearVSandwich}
  k + g - q_\infty \lt v \lt g;
\end{equation}
in particular, \(v\) is continuous on \(\partial \R^2_+ \setminus \inter K\).
By the strong maximum principle, we have \(v > k + g - q_\infty\) in \(\R^2_+ \setminus K\).

Moreover, we have by \eqref{eqLinearVSandwich} for every \(x \in \R^2_+ \setminus B (0, 2 R)\),
\[
  v (x) \gt - \frac{W}{\alpha + 1} x_1^{\alpha + 1}.
\]
Define
\begin{equation}
\label{eqLinearw}
  w (x) = -  W x_1^{\alpha + 1}\frac{(2 R)^{\alpha + 2}}{\abs{x}^{\alpha + 2}}.
\end{equation}
One checks that \(\dive \frac{\nabla w}{b} = 0\)
and \(w \lt v\) on \(\partial B (0, 2 R)\). By a comparison argument, we have thus that 
\(w \lt v\) in \(\R^2_+ \setminus B (0, 2 R)\).
In particular,
\[
 \lim_{\abs{x} \to \infty} \frac{ v (x) }{q^\infty (x)} = 0.
\]
Finally, note that if \(x \in \R^2_+ \setminus B (0, 2 R)\), by combining a classical estimate \cite{GiTr01}*{Corollary 6.3} with \eqref{eqLinearw}:
\[
 \frac{\abs{\nabla v (x)}}{x_1^{\alpha}} \lt \frac{C}{x_1^{\alpha + 1}} \sup_{y \in B (x, x_1 / 2)} \abs{v (y)}
\lt C' \frac{R^{\alpha + 2}}{\abs{x}^{\alpha + 2}},
\]
and thus 
\[
 \lim_{\abs{x} \to \infty} \frac{\abs{\nabla v (x)} }{x_1^{\alpha}} = 0.\qedhere
\]
\end{proof}

\begin{proof}[Proof of theorem~\ref{theoremVortexRingOutsideCompact}]
Since \(K\) is simply connected \(\partial (\R^2_+ \setminus K)\) is connected and \(\psi (r, z) = k\) on \(\partial (\R^2_+ \setminus K)\) for some \(k < 0\).
Defining \(q = -\psi\) and \(q^\infty (x) = \frac{W}{2} x_1^2 + k\), we observe that \(q\) is also the solution given by lemma~\ref{lemmaLinear}.
We are going to apply proposition~\ref{propositionStrict}. 
We observe that by proposition~\ref{propositionExistenceInvariant} and proposition~\ref{ThUpperBound}, we have
\[
 \lim_{\varepsilon \to 0} \frac{c^\infty_\varepsilon}{\log \tfrac{1}{\varepsilon}} =  
\inf_{(r, z) \in \R_+^2} \frac{q^\infty(r, z)^2}{r}.
\]
By a direct computation,
\[
  \inf_{\R^2} \frac{(q^\infty)^2}{b} = \frac{q^\infty (r_*, z)^2}{r_*}.
\]
Since \(K\) is compact, there exists \(z_* \in \R\) such that \((r_*, z_*) \not \in K\). By proposition~\ref{ThUpperBound}, lemma~\ref{lemmaLinear} and  proposition~\ref{propositionAsymptotics}
\[
   \limsup_{\varepsilon \to \infty} \frac{c_\varepsilon}{\log \tfrac{1}{\varepsilon}} \lt \pi \frac{q(r^\infty_*, z)^2}{r^\infty_*} < \limsup_{\varepsilon \to \infty} \frac{c^\infty_\varepsilon}{\log \tfrac{1}{\varepsilon}}.
\]
By proposition~\ref{propositionStrict}, a solution \(u_\varepsilon\) exists if \(\varepsilon\) is small enough. One defines the associated flow and studies its asymptotics as in the proof of theorem~\ref{theoremVortexRingSpace}.
\end{proof}

The question of where the vortex concentrates gives rise to a result depending on the geometry of the compact set \(D\):
\begin{proposition}
If \(k\) is sufficiently large and \(\alpha > 0\), then 
\[
  \inf_{x \in \partial (\R^2_+ \setminus K)} \frac{q (x)^2 }{x_1^\alpha} <
\inf_{x \in \R^2_+ \setminus K} \frac{q (x)^2 }{x_1^\alpha}.
\]
\end{proposition}

\begin{proof}
First one has 
\[
 \inf_{x \in \partial (\R^2_+ \setminus K)} \frac{q (x)^2 }{x_1^\alpha}
= k^2 \inf_{x \in \partial (\R^2_+ \setminus K)} \frac{1}{x_1^\alpha}.
\]
Since \(K\) is compact, there exists \(R > 0\) such that \(K \subset B (0, R)\). Take \(a \in \R\) such that \(\abs{a} \gt R\). One has 
\[
 \inf_{x \in \R^2_+ \setminus K} \frac{q (x, z)}{x_1^\alpha}
\lt  \inf_{x \in \R_+} \frac{q (x)^2}{x_1^\alpha}
\lt  \inf_{x \in \R_+} \frac{q^\infty (x)^2}{x_1^\alpha}
= 4 \Bigl(k \frac{\alpha + 1}{2 \alpha + 1}\Bigr)^\frac{\alpha + 2}{\alpha + 1} W^\frac{\alpha}{\alpha + 1}.\qedhere
\]
\end{proof}

\subsection{Vortices for the shallow water equation}
We finish by sketching the proofs for the shallow water equation:

\begin{proof}[Proof of theorem~\ref{theoremLakeMaximumDepth}]
Set for \(x \in \Omega\), \(q (x) = \frac{\kappa}{2\pi} \sup_\Omega b\). 
By proposition~\ref{propositionExistenceBounded}, \eqref{problemP} has a solution \(u_\varepsilon\). Define for \(x \in \Omega\)
\[
 \mathbf{v}_\varepsilon (x) = \curl u_\varepsilon(x)
\]
and 
\[
 h (x) = \frac{1}{\varepsilon^2} \frac{ \bigl(u_\varepsilon (x) - q_\varepsilon (x) \bigr)_+^{p + 1}}{p + 1} - \frac{\abs{\mathbf{v}_\varepsilon (x)}^2}{2}.
\]
One checks directly that this is a steady flow of the shallow water equation \eqref{eqLake} and that 
\[
 \curl \mathbf{v}_\varepsilon (x) = \frac{1}{\varepsilon^2} \bigl(u_\varepsilon (x) - q_\varepsilon (x) \bigr)_+^{p},
\]
and that 
\[
\inf_\Omega \frac{q^2}{b} = \Bigl(\frac{\kappa}{2 \pi}\Bigr)^2  \sup_\Omega b,
\]
so that \(\curl \mathbf{v}_\varepsilon\) has the required asymptotic properties by proposition~\ref{propositionAsymptotics}.
\end{proof}

\begin{proof}[Proof of theorem~\ref{theoremLakeStrongWind}]
Set for \(x \in \Omega\), \(q (x) = - \psi_0 (x)\). 
By proposition~\ref{propositionExistenceBounded}, \eqref{problemP} has a solution \(u_\varepsilon\). Define for \(x \in \Omega\)
\[
 \mathbf{v}_\varepsilon (x) = \curl (u_\varepsilon - q_\varepsilon)
\]
and 
\[
 h (x) = \frac{1}{\varepsilon^2} \frac{ \bigl(u_\varepsilon (x) - q_\varepsilon (x) \bigr)_+^{p + 1}}{p + 1} - \frac{\abs{\mathbf{v}_\varepsilon (x)}^2}{2}.
\]
One checks directly that this is a steady flow of the shallow water equation \eqref{eqLake} and that \(\curl \mathbf{v}_\varepsilon\) has the required asymptotic properties.
\end{proof}

\end{document}